\documentclass[11pt]{amsart}
\usepackage{amsmath,amssymb,amsthm}
\usepackage{amsmath,amssymb,amsthm,amscd}
\usepackage[frame,cmtip,arrow,matrix,line,graph,curve]{xy}
\usepackage{graphpap, color}
\usepackage[mathscr]{eucal}
\usepackage{color}
\usepackage{verbatim}

\def\dfrac{\displaystyle\frac}
\def\dsum{\displaystyle\sum}

\let\oldsection\section
\renewcommand\section{\setcounter{equation}{0}\oldsection}

\newtheorem{theo}{\indent Theorem}[section]
\newtheorem{lemma}{\indent Lemma}[section]

\newtheorem{rema}{\indent Remark}[section]

\newcommand{\al}{\alpha}
\newcommand{\la}{\lambda}
\newcommand{\abc}[1]{\left( #1 \right)}%
\newcommand{\abz}[1]{\left[ #1 \right]}
\renewcommand{\leq}{\leqslant}
\renewcommand{\geq}{\geqslant}

\numberwithin{equation}{section}

\begin{document}

	\title{Pogorelov type $C^2$ estimates for Sum Hessian equations}
	
	\author{Pengfei Li}
	\address{School of Mathematical Science\\
		Jilin University\\ Changchun\\ China}
	\email{lipf22@mails.jlu.edu.cn}
	\author{Changyu Ren}
	\address{School of Mathematical Science\\
		Jilin University\\ Changchun\\ China} \email{rency@jlu.edu.cn}
	\begin{abstract}
		In this paper, We establish Pogorelov type $C^2$ estimates for  the admissible solutions with $\sigma_{k}(D^2u)$ bounded from below of Sum Hessian equations. We also proved the lower bounded condition can be removed when $k=n$.
	\end{abstract}
	\maketitle
	\tableofcontents

	\section{Introduction}
	In this paper, I will mainly study the Pogorelov  type $C^{2}$ estimate of the solutions to the Dirichlet problem of the following sum Hessian equation:
	\begin{equation}\label{1.1}
		\begin{cases}
			\sigma_{k - 1}(D^{2}u)+\alpha\sigma_{k}(D^{2}u)=f(x,u,Du), & \text{in }\Omega, \\
			u = 0, & \text{on }\partial\Omega.
		\end{cases}
	\end{equation}
	Where $u$ is a function defined on the domain $\Omega$, $Du$ is the gradient of $u$, $D^{2}u$ is the Hessian matrix of $u$, $\alpha\geq0$ is a constant, and $f\geq m > 0$ is a given smooth function. $\sigma_{k}(D^{2}u)$ represents the $k$-th elementary symmetric polynomial of the eigenvalues of the matrix $D^{2}u$. That is, for $\lambda = (\lambda_{1},\cdots,\lambda_{n})\in\mathbb{R}^{n}$,
	\[
	\sigma_{k}(\lambda)=\sum_{1\leq i_{1}<\cdots <i_{k}\leq n}\lambda_{i_{1}}\cdots\lambda_{i_{k}}.
	\]

	When $\alpha=0$, the equation \eqref{1.1} is classical $k$-Hessian equation, that is
	\begin{equation}\label{1.2}
		\sigma_{k}(D^{2}u)=f(x,u,Du),\quad x\in\Omega.
	\end{equation}
	The equation \eqref{1.2} and the following curvature equation
	\begin{equation}\label{1.3}
		\sigma_{k}(\kappa(X))=\psi(X,\nu),\quad X\in M,
	\end{equation}
	are important research contents in the fields of fully nonlinear partial differential equations and geometric analysis. When $k = 1$, equations \eqref{1.2} and \eqref{1.3} are respectively the semilinear elliptic equation and the prescribed mean curvature equation. When $k = 2$, equation \eqref{1.3} is the prescribed scalar curvature equation. When $k = n$, equations \eqref{1.2} and \eqref{1.3} are respectively the Monge - Ampère equation and the prescribed Gaussian curvature equation. For general $k$, $\sigma_{k}(\kappa(X))$ represents the Weingarten curvature at $X$.
	
	An important problem in the study of the $k$-Hessian equation \eqref{1.2} and the curvature equation is how to obtain the $C^{2}$ estimate and the curvature estimate of the solution. There are many studies on this aspect, such as references \cite{CNS1, CNS3, GB1, JL, GG, Ivo2, Ivo1, GLL, GLM, Y1, GRW, Chu, SX, RW1, RW2, Lu}.

	The Pogorelov  type $C^{2}$ estimate is an interior $C^{2}$ estimate with boundary information. Pogorelov first established this estimate for the Monge-Ampère equation \cite{P1}. Liu-Trudinger \cite{LT} and Jiang-Trudinger \cite{JT} established the Pogorelov  type estimate for more general Monge-Ampère type equations. Yuan \cite{YY} used the monotonicity method to prove the interior estimate and the Pogorelov  type estimate of the Monge-Ampère equation. Sheng-Urbas-Wang \cite{SUW} established the Pogorelov  type estimate for a large class of curvature equations including the Hessian equation. Chou-Wang \cite{CW} and Wang \cite{Wang} established the Pogorelov  type estimate for the $k$-convex solution of the equation
	\[
	\begin{cases}
		\sigma_{k}(D^{2}u)=f(x,u), & \text{in }\Omega, \\
		u = 0, & \text{on }\partial\Omega.
	\end{cases}
	\]
	that is,
	\[
	(-u)^{1+\varepsilon}\Delta u\leq C,
	\]
	where $\varepsilon > 0$ can be an arbitrarily small positive number. When the right hand function of  this equation depends on the gradient, that is,
	\begin{equation}\label{1.4}
		\begin{cases}
			\sigma_{k}(D^{2}u)=f(x,u,Du), & \text{in }\Omega, \\
			u = 0, & \text{on }\partial\Omega.
		\end{cases}
	\end{equation}
	Li-Ren-Wang \cite{LiRW} established the Pogorelov  type estimate for the $k + 1$-convex solution
	\[
	(-u)\Delta u\leq C.
	\]
	In particular, reference \cite{LiRW} also proved that for the 2-convex solution of the $\sigma_{2}$ equations, there is a Pogorelov type estimate
	\begin{equation}\label{1.5}
		(-u)^{\beta}\Delta u\leq C,
	\end{equation}
	especially, the power $\beta$ of $u$ depends on $\sup\limits_{\Omega}|Du|$. Chen-Xiang\cite{CX} further proved that if $f = 1$ and
	$
	\sigma_{3}(D^{2}u) > - A,$
	then the power $\beta$ can be independent of $\sup\limits_{\Omega}|Du|$ for the $\sigma_{2}$ equations . Chen-Tu-Xiang \cite{CTX} established the Pogorelov  type estimate of the form \eqref{1.5} for the semi-convex admissible solution for the equation \eqref{1.4}. Jiao \cite{JJ} studied the Pogorelov estimate of the degenerate curvature equations.

	In addition to the Hessian equation, the Hessian-type equation has also attracted extensive attention, including the sum Hessian equation, the Hessian quotient equation, etc. The sum Hessian operator refers to the linear combination of Hessian operators of various orders. For example, the following equation proposed by Harvey and Lawson \cite{HLB} in the study of minimal submanifold problems
	$$
	{\rm Im}\det(\delta_{ij}+iu_{ij})=\sum_{k = 0}^{[(n - 1)/2]}(- 1)^{k}\sigma_{2k + 1}(\lambda(u_{ij})) = 0.
	$$
	Krylov \cite{K1} and Dong \cite{D1} studied the following nonlinear equation
	\[
	P_{m}(u_{ij})=\sum_{k = 0}^{m - 1}(l_{k}^{+})^{m - k}(x)\sigma_{k}(u_{ij})=g^{m - 1}(x).
	\]
	Li-Ren-Wang \cite{LRW} studied the concavity of the operators $\sum_{s = 0}^{k}\alpha_{s}\sigma_{s}$ and $\sigma_{k}+\alpha\sigma_{k - 1}$, and established the curvature estimate of the convex solution of the following sum Hessian equation
	\[
	\sum_{s = 0}^{k}\alpha_{s}\sigma_{s}(\kappa(X))=f(X,\nu),\quad X\in M.
	\]
	Guan-Zhang \cite{GZ}, Chen-Lu-Tu-Xiang \cite{CLTX}, Sheng-Xia \cite{SXia} and Zhou \cite{Zhou2} also studied other types of sum Hessian equations.
	 Liu-Ren \cite{LR} studied the sum Hessian equation of the form \eqref{1.1},
	when $k = 2,3$, the $k - 1$ convex solutions has Pogorelov  type estimates, but the power of $u$ is very large. If the right hand function $f^{\frac{1}{k}}(x,u,p)$ is convex with respect to $p$, the power can be reduced to $1+\varepsilon$. If the convexity is strengthened to $k$-convex, the power can be completely reduced to 1. 
	
	For other forms of Hessian-type equations, 
	References \cite{CDH,CTX2,LMZ,Mei,Zhou1} studied the Hessian quotient type equation.
	References \cite{Bao,HPX,HSX} studied the parabolic $k$-Hessian equation and established the Pogorelov  type estimates. 	Chu-Jiao \cite{JC} and Dong\cite{DWS2,DWS1} studied the global $C^{2}$ estimate and the Pogorelov  type estimate of the other Hessian type equations.
	
	As we all know, the admissible solution of the $\sigma_{k}$ equation is the G$\mathring{a}$rding cone
	$$\Gamma_k=\left\{\lambda\in\mathbb{R}^n|\sigma_1(\lambda)>0,\cdots,\sigma_k(\lambda)>0\right\}.$$
	We denote
	\begin{equation*}
		S_k(\la):=\sigma_{k-1}(\la)+\alpha\sigma_{k}(\la).
	\end{equation*}
	In Li-Wang-Ren\cite{LRW}, the author
	proved that the admissible solution set of  the sum Hessian equation \eqref{1.1} is
	$$\left\{\lambda\in\mathbb{R}^n|S_1(\lambda)>0,\cdots,S_k(\lambda)>0\right\}.$$
	According to reference \cite{LR}, the set is equal to this set 
	$$
	\tilde\Gamma_k=\Gamma_{k-1}\cap\{\la\in\mathbb{R}^n|S_k(\lambda)>0\}.
	$$
	The main results in the paper are as follows:
	\begin{theo}
		Let $\Omega$ be a bounded domain in $\mathbb{R}^n$, $f(x, u, p) \in C^2(\overline{\Omega} \times \mathbb{R} \times \mathbb{R}^n)$ and $f\ge m>0$, $u \in C^4(\Omega) \cap C^2(\overline{\Omega})$ is a $k - 1$ convex solution of the equation \eqref{1.1}.
	 If there exists a constant $G>0$ such that $\sigma_k(D^2u)\geq - G$, then
		\[
		(-u)^\beta\Delta u\leq C.
		\]
		Here the constants $\beta$ and $C$ depend on $n, k, \Omega, \alpha, G,\|u\|_{C^1}, f$.
		\end{theo}
	\begin{theo}
			Under the conditions of Theorem 1, if $k=n$, then for all $n-1$ convex solutions of equation \eqref{1.1} (without the condition that $\sigma_n$ is bounded from below), we have
		\[
		(-u)^\beta\Delta u\leq C.
		\]
		Here the constants $\beta$ and $C$ depend on $n, k, \Omega, \alpha, \|u\|_{C^1}, f$.
	\end{theo}

	\section{Preliminary}
	\par
	In this section, we list some useful preliminary knowledge.
	\begin{lemma}\label{lem2.1}
		Let $\lambda=(\lambda_1,\cdots,\lambda_n)\in
		\mathbb{R}^n$, we have\\
		\par
		(i) $S_k^{pp}(\lambda):=\dfrac{\partial
			S_k(\lambda)}{\partial\lambda_p}=S_{k-1}(\lambda|p)$,
		\quad $p=1,2,\cdots,n$;
		\par
		(ii)
		$S_k^{pp,qq}(\lambda):=\dfrac{\partial^2S_k(\lambda)}{\partial\lambda_p\partial\lambda_q}=S_{k-2}(\lambda|pq)$,
		\quad $p,q=1,2,\cdots,n$, \quad and $S_k^{pp,pp}(\la)=0$;
		\par
		(iii) $S_k(\lambda)=\lambda_iS_{k-1}(\lambda|i)+S_k(\lambda|i),$
		$i=1,\cdots,n$;
		\par
		(iv)
		$\dsum_{i=1}^nS_k(\lambda|i)=(n-k)S_k(\lambda)+\sigma_{k-1}(\lambda);$
		\par
		(v)
		$\dsum_{i=1}^n\lambda_iS_{k-1}(\lambda|i)=kS_k(\lambda)-\sigma_{k-1}(\lambda).$
	\end{lemma}
	\begin{proof}
		See the basic properties of sum Hessian operator in \cite{LR}.
	\end{proof}
	The following Lemma comes from \cite{RW2}.
	\begin{lemma}\label{lem2.2}
		For $\lambda=(\lambda_1,\cdots,\lambda_n)\in{\Gamma}_k$, $\lambda_1\ge\cdots\ge\lambda_n.$
		\par
		(i) Newton's inequality
		$$\left(\frac{\sigma_{k}(\lambda)}{C_n^k}\right)^2\le\left(\frac{\sigma_{k-1}(\lambda)}{C_n^{k-1}}\right)\left(\frac{\sigma_{k+1}(\lambda)}{C_n^{k+1}}\right),$$
		\par
		furthermore
		$$\sigma_{k}^2(\lambda)-\sigma_{k-1}(\lambda)\sigma_{k+1}(\lambda)\ge\Theta\sigma_{k}^2(\lambda),\quad\Theta=1-\frac{C_n^{k-1}C_n^{k+1}}{(C_n^k)^2}.$$
		\par
		(ii)
		$$\sigma_{k}(\lambda)\le C_{n,k}\lambda_1\cdots\lambda_k.$$
		\par
		(iii)$$\sigma_{l}\ge \lambda_1\cdots\lambda_l,\quad l=1,2,\cdots,k-1.$$
		\par
		(iv) If $\lambda_i\le 0$, then
		$$-\lambda_i\le \frac{n-k}{k}\la_1.$$
		\par
		(v)
		$$\lambda_{k}+\lambda_{k+1}+\cdots+\la_n>0,$$
		$$|\la_i|\le n\la_k, \quad\forall i>k.$$
	\end{lemma}
	
	\begin{lemma}\label{lem2.3}	For $\lambda=(\lambda_1,\cdots,\lambda_n)\in\widetilde{\Gamma}_k$, $\lambda_1\ge\cdots\ge\lambda_n.$
		\par
		(i) 
		$$S_l(\lambda)>\frac{1}{2}(\la_1\cdots\la_{l-1}+\al\la_1\cdots\la_l), \quad l=1,2,\cdots,k-1.$$
		\par
		(ii)
		If $i=1,2,\cdots,k-1$, then exist a positive constant $\theta$ depending on $n$, $k$, such that
		$$S_k^{ii}\ge\frac{\theta S_k(\la)}{\la_i}.$$
	\end{lemma}
	\begin{proof}
		See Lemma 2.3 in \cite{LR}.
	\end{proof}
	
	\begin{lemma}\label{lem2.4}
		Assume that $k>l$, for $\vartheta=\dfrac{1}{k-l}$, then for
		$\la\in\tilde\Gamma_k$, we have
		\begin{align*}
			-&\frac{S_k^{pp,qq}}{S_k}u_{pph}u_{qqh}+\dfrac{S_l^{pp,qq}}{S_l}u_{pph}u_{qqh}\\
			&\geq \abc{\dfrac{(S_k)_h}{S_k}-\dfrac{(S_l)_h}{S_l}}
			\abc{(\vartheta-1)\dfrac{(S_k)_h}{S_k}-(\vartheta+1)\dfrac{(S_l)_h}{S_l}}.\nonumber
		\end{align*}
		Furthermore, for sufficiently small $\delta>0$, we have
		\begin{align*}
			-&S_k^{pp,qq}u_{pph}u_{qqh} +(1-\vartheta+\dfrac{\vartheta}{\delta})\dfrac{(S_k)_h^2}{S_k}\\
			&\geq S_k(\vartheta+1-\delta\vartheta) \abz{\dfrac{(S_l)_h}{S_l}}^2
			-\dfrac{S_k}{S_l}S_l^{pp,qq}u_{pph}u_{qqh}.\nonumber
		\end{align*}
	\end{lemma}
	
	\begin{proof}
		See Lemma 2.4 in \cite{LR}.
	\end{proof}
	The following Lemma comes from \cite{Ball}.
	
	\begin{lemma} \label{lem2.5}
		Denote by $Sym(n)$ the set of all $n\times n$ symmetric matrices.
		Let $F$ be a $C^2$ symmetric function defined in some open subset
		$\Psi \subset Sym(n)$. At any diagonal matrix $A\in\Psi$
		with distinct eigenvalues, let$\ddot{F}(B,B)$ be the second derivative of $C^2$ symmetric function $F$ in direction $B \in
		Sym(n)$, then
		\begin{align*}
			\ddot{F}(B,B) =  \sum_{j,k=1}^n {\ddot{f}}^{jk}
			B_{jj}B_{kk} + 2 \sum_{j < k} \frac{\dot{f}^j -
				\dot{f}^k}{{\kappa}_j - {\kappa}_k} B_{jk}^2.
		\end{align*}
	\end{lemma}
	
	\begin{lemma}\label{lem2.6}
		Assume that $\lambda=(\lambda_1,\cdots,\lambda_n)\in \widetilde{\Gamma}_k$, $\lambda_1\geq\cdots\geq\lambda_n$ and exist constants $F,G\ge0$, such that $\sigma_{k-1}+\alpha\sigma_{k}\le F$, $\sigma_{k}\ge -G$, then
		$$\lambda_n\ge -K,$$
		where $K$ is a positive constant depending only on $n,k,F, G,\al.$
	\end{lemma}
	\begin{proof}
		Without loss of generality, we assume that $\lambda_n<0.$\\
		\textbf{Case1}:$\lambda_{k-1}\le 1$.\\
		By $\lambda\in \Gamma_{k-1}$, we have
		$$\sigma_{k-1}^{11,\cdots ,k-2,k-2}(\lambda)=\lambda_{k-1}+\cdots+\lambda_n>0,$$
		hence
		$$\lambda_n>-\lambda_{k-1}-\cdots-\lambda_{n-1}\ge k-n-1. $$
		\textbf{Case2}:$\lambda_{k-1}>1$.\\
		Since 
		$$\sigma_{k-1}(\lambda|n)=\sigma_{k-1}(\lambda)-\lambda_n\sigma_{k-2}(\lambda|n)>0,$$
		we have $(\lambda|n)\in\Gamma_{k-1}$.\\
		Since
		$$\sigma_{k}(\lambda)=\sigma_{k}(\lambda|n)+\lambda_n\sigma_{k-1}(\lambda|n)>-G,$$
		we get
		\begin{equation}\label{2.4}
			\lambda_n>\frac{-G}{\sigma_{k-1}(\lambda|n)}-\frac{\sigma_{k}(\lambda|n)}{\sigma_{k-1}(\lambda|n)}.		
		\end{equation}
		Due to
		\begin{align*}
			&\sigma_{k-1}(\lambda)+\alpha\sigma_{k}(\lambda)\\
			=&\lambda_n[\sigma_{k-2}(\lambda|n)+\alpha\sigma_{k-1}(\lambda|n)]+\sigma_{k-1}(\lambda|n)+\alpha\sigma_{k}(\lambda|n)\le F,
		\end{align*}
		thus
		\begin{equation}\label{2.5}
			\lambda_n\le\frac{F}{\sigma_{k-2}(\lambda|n)+\alpha\sigma_{k-1}(\lambda|n)}-\frac{\sigma_{k-1}(\lambda|n)+\alpha\sigma_{k}(\lambda|n)}{\sigma_{k-2}(\lambda|n)+\alpha\sigma_{k-1}(\lambda|n)}.
		\end{equation}
		Using \eqref{2.4}, \eqref{2.5}, we obtain
		$$\frac{F}{\sigma_{k-2}(\lambda|n)+\alpha\sigma_{k-1}(\lambda|n)}-\frac{\sigma_{k-1}(\lambda|n)+\alpha\sigma_{k}(\lambda|n)}{\sigma_{k-2}(\lambda|n)+\alpha\sigma_{k-1}(\lambda|n)}\ge\frac{-G}{\sigma_{k-1}(\lambda|n)}-\frac{\sigma_{k}(\lambda|n)}{\sigma_{k-1}(\lambda|n)}.$$
		Combining Lemma 2.2 (i), we have
		\begin{align*}
			F&\ge\frac{\sigma_{k-1}^2(\lambda|n)-\sigma_{k-2}(\lambda|n)\sigma_{k}(\lambda|n)}{\sigma_{k-1}(\lambda|n)}-G\frac{\sigma_{k-2}(\lambda|n)}{\sigma_{k-1}(\lambda|n)}-G\alpha\\
			&\ge \Theta\sigma_{k-1}(\lambda|n)-G\frac{\sigma_{k-2}(\lambda|n)}{\sigma_{k-1}(\lambda|n)}-G\alpha\\
			&\ge \Theta\sigma_{k-1}(\lambda|n)-G\frac{\sigma_{k-2}(\lambda|n)}{\sigma_{k-1}(\lambda)-\lambda_n\sigma_{k-2}(\lambda|n)}-G\alpha\\
			&\ge \Theta\sigma_{k-1}(\lambda|n)-\frac{G}{-\la_n}-G\alpha.
		\end{align*}
		Since$(\lambda|n)\in\Gamma_{k-1}$, using Lemma 2.2 (iii), we have
		$$\sigma_{k-1}(\lambda|n)=\sigma_{k-1}(\lambda)-\lambda_n\sigma_{k-2}(\lambda|n)\ge -\la_n\lambda_1\cdots\lambda_{k-2}\ge-\la_n$$
		hence
		$$\Theta(-\la_n) -\frac{G}{-\la_n}\le F+G\al,$$
		which implies that $-\la_n\le K.$
	\end{proof}

	\section{Two concavity inequalities}
	\subsection{A concavity inequality about $S_k$}
	In this section, we will prove the following concavity inequality for  Sum Hessian operator. The lemma is inspired by \cite{Lu}.
	\begin{lemma}\label{3.4}
		Let $\lambda=(\lambda_1,\cdots,\lambda_n)\in \widetilde{\Gamma}_k$ with $\lambda_1\ge\cdots\ge\lambda_n$ and let $1\le l \le k-2$. Then $\forall \epsilon,\delta,\delta_0 \in (0,1)$, $\exists \delta^{'} >0$, such that if $\lambda_l \ge \delta\lambda_1,\lambda_{l+1} \le \delta^{'}\lambda_1$, we have
		$$-\sum\limits_{p\ne q}\frac{S_k^{pp,qq}\xi_p \xi_q}{S_k}+(1-\vartheta+\frac{\vartheta}{\delta})\frac{\left(\sum\limits_{i}S_k^{ii}\xi_i \right)^2}{S_k^2}\ge (1+\vartheta-\delta\vartheta-\epsilon)\frac{\xi_1^2}{\lambda_1^2}-\delta_0 \sum\limits_{i>l}\frac{S_k^{ii}\xi_i^2}{\lambda_1 S_k},$$
		where $\vartheta=\frac{1}{k-l}$, $\xi=(\xi_1,\cdots,\xi_n)$ is an arbitrary vector in $\mathbb{R}^n.$
	\end{lemma}
	\begin{proof}
		If $\al=0$, that is $S_k=\sigma_{k-1}$, we just need to modify Lemma 3.1 in \cite{LSY}, so we assume $\al>0$. By Lemma \ref{lem2.4}, we have
		\begin{align}\label{3.1}
			&-\sum\limits_{p\ne q}\frac{S_k^{pp,qq}\xi_p\xi_q}{S_k}+(1-\vartheta+\frac{\vartheta}{\delta})\frac{\left(\sum\limits_{i}S_k^{ii}\xi_i \right)^2}{S_k^2}\notag\\
			\ge&(1+\vartheta-\delta\vartheta)\left(\frac{\sum\limits_iS_l^{ii}\xi_i}{S_l}\right)^2-\sum\limits_{p\ne q} \frac{S_l^{pp,qq}\xi_p\xi_q}{S_l}\notag\\
			\ge&\frac{1}{S_l^2}\left((1+\vartheta-\delta\vartheta)\sum\limits_i(S_l^{ii}\xi_i)^2+\sum\limits_{p\ne q}(S_l^{pp}S_l^{qq}-S_lS_l^{pp,qq})\xi_p\xi_q\right)
		\end{align}
		We claim that
		\begin{equation}\label{3.2}
			\sum\limits_{p\neq q}(S_l^{pp}S_l^{qq}-S_lS_l^{pp,qq})\xi_p\xi_q\ge-\displaystyle\frac{\epsilon}{2}\sum\limits_{i \le l}(S_l^{ii}\xi_i)^2-\frac{C}{\epsilon}\sum\limits_{i>l}(S_l^{ii}\xi_i)^2.
		\end{equation}
		For $l=1$, we have $S_1=\sigma_0+\alpha\sigma_1=1+\alpha(\lambda_1+\cdots+\lambda_n)$, so 
		\eqref{3.2} become 
		$$\sum\limits_{p\neq q}\xi_p\xi_q\ge-\displaystyle\frac{\epsilon}{2}\xi_1^2-\frac{C}{\epsilon}\sum\limits_{i>1}\xi_i^2.$$
		Since
		\begin{align}\label{bds1}
			\sum\limits_{p\neq q}|\xi_p\xi_q|
			&=2\sum_{i=2}^n|\xi_1\xi_i|+\sum_{p\ne q;p,q\ne 1}|\xi_p\xi_q|\notag\\
			&\le\sum_{i=2}^n\left(\frac{\epsilon}{2(n-1)}\xi_1^2+\frac{2(n-1)}{\epsilon}\xi_i^2\right)+C\sum_{i>1}\xi_i^2\notag\\
			&\le \frac{\epsilon}{2}\xi_1^2+\frac{C}{\epsilon}\sum_{i>1}\xi_i^2,
		\end{align}
		we have
		$$	\sum\limits_{p\neq q}\xi_p\xi_q\ge 	-\sum\limits_{p\neq q}|\xi_p\xi_q|\ge -\displaystyle\frac{\epsilon}{2}\xi_1^2-\frac{C}{\epsilon}\sum\limits_{i>1}\xi_i^2. $$
		Thus the claim \eqref{3.2} holds for $l=1$.\\
		For $2\le l\le k-2$, following the formulas (4.15-4.21) in \cite{LR}, we can also get the claim \eqref{3.2}.
		
		By \eqref{3.1} and \eqref{3.2}, we have
		\begin{align}\label{3.9}
			&-\sum\limits_{p\ne q}\frac{S_k^{pp,qq}\xi_p\xi_q}{S_k}+(1-\vartheta+\frac{\vartheta}{\delta})\frac{\left(\sum\limits_{i}S_k^{ii}\xi_i \right)^2}{S_k^2}\notag\\
			\ge&\frac{1}{S_l^2} \left[(1+\vartheta-\delta\vartheta)\sum\limits_{i\le l}\left(S_l^{ii}\xi_i \right)^2-\displaystyle\frac{\epsilon}{2}\sum\limits_{i \le l}(S_l^{ii}\xi_i)^2-\frac{C}{\epsilon}\sum\limits_{i>l}(S_l^{ii}\xi_i)^2 \right]\notag\\
			\ge&\frac{1}{S_l^2}\left(1+\vartheta-\delta\vartheta-\frac{\epsilon}{2}\right)\left(S_l^{11}\xi_1\right)^2-\frac{C}{\epsilon}\sum\limits_{i>l}\left(\frac{S_l^{ii}}{S_l}\right)^2\xi_i^2.
		\end{align}
		For $p\le l$, by choosing $\delta^{'}$ sufficiently small and $\lambda_1$ sufficiently large, we have
		\begin{align*}
			S_l(\lambda |p)&=\sigma_{l-1}(\lambda|p)+\alpha\sigma_{l}(\lambda|p)\le C\frac{\lambda_1\cdots\lambda_{l}}{\lambda_p}+C\alpha\frac{\lambda_1\cdots\lambda_{l+1}}{\lambda_p}\\
			&\le C\frac{\lambda_1\cdots\lambda_l(1+\al\lambda_{l+1})}{\lambda_p}\le C\frac{S_l(1+\al\delta^{'}\lambda_1)}{\delta\lambda_1}\le C\frac{\delta^{'}}{\delta}S_l,
		\end{align*}
		hence,
		$$\lambda_pS_l^{pp}=S_l-S_l(\lambda|p)\ge(1-\frac{C\delta^{'}}{\delta})S_l, \forall p\le l.$$
		So we have
		\begin{align}\label{3.10}
			\left(1+\vartheta-\delta\vartheta-\frac{\epsilon}{2}\right)(S_l^{11})^2&\ge\left(1+\vartheta-\delta\vartheta-\frac{\epsilon}{2}\right)(1-\frac{C\delta^{'}}{\delta})^2\frac{ S_l^2}{\lambda_1^2}\notag\\
			&\ge(1+\vartheta-\delta\vartheta-\epsilon)\frac{ S_l^2}{\lambda_1^2},
		\end{align}
		by choosing $\delta^{'}$ sufficiently small.\\
		Using Lemma 2.2 (ii) and Lemma 2.3 (i), we have $\forall i>l$,
		$$S_l^{ii}=\sigma_{l-2}(\lambda|i)+\alpha\sigma_{l-1}(\lambda|i)\le C\lambda_1\cdots\lambda_{l-2}+C\alpha\lambda_1\cdots\lambda_{l-1} ,  $$
		$$S_l\ge \frac{1}{2}(\lambda_1\cdots\lambda_{l-1}+\alpha\lambda_1\cdots\lambda_{l}).$$
		Hence, if $i>l$,
		\begin{align}\label{3.11}
			\left(\frac{S_l^{ii}}{S_l}\right)^2&\le
			C\left(\frac{\lambda_1\cdots\lambda_{l-2}+\alpha\lambda_1\cdots\lambda_{l-1}}{\frac{1}{2}(\lambda_1\cdots\lambda_{l-1}+\alpha\lambda_1\cdots\lambda_{l})}\right)^2\notag\\&\le C\left(\frac{1+\lambda_{l-1}}{\lambda_{l-1}+\lambda_{l-1}\lambda_{l}}\right)^2\le C\left(\frac{2\lambda_{l-1}}{\lambda_{l-1}\lambda_{l}}\right)^2\le\frac{C}{\lambda_{l}^2}\le\displaystyle\frac{C}{\delta^2\lambda_1^2}.
		\end{align}
		By Lemma 2.3 (ii), $\forall i>l$, if $l<i\le k-1$,
		$$\lambda_1 S_k^{ii} \ge \lambda_1\frac{\theta S_k}{\lambda_i}\ge\lambda_1\frac{\theta S_k}{\lambda_{l+1}}\ge \frac{\theta S_k}{\delta^{'}},$$
		if $i\ge k$,
		$$\lambda_1 S_k^{ii}\ge \lambda_1S_k^{k-1,k-1} \ge \lambda_1\frac{\theta S_k}{\lambda_{k-1}}\ge\lambda_1\frac{\theta S_k}{\lambda_{l+1}}\ge \frac{\theta S_k}{\delta^{'}}.$$
		So choosing $\delta^{'}$ sufficiently small, we have $\forall i>l$,
		\begin{equation}\label{3.12}
			\frac{C}{\epsilon\delta^2 \delta_0}S_k\le \lambda_1 S_k^{ii}.
		\end{equation}
		Finally, plugging \eqref{3.10}, \eqref{3.11} and \eqref{3.12} into \eqref{3.9}, we have
		\begin{align*}
			&-\sum\limits_{p\ne q}\frac{S_k^{pp,qq}\xi_p\xi_q}{S_k}+(1-\vartheta+\frac{\vartheta}{\delta})\frac{\left(\sum\limits_{i}S_k^{ii}\xi_i \right)^2}{S_k^2}\\
			\ge&(1+\vartheta-\delta\vartheta-\epsilon)\frac{\xi_1^2}{\lambda_1^2}-\frac{C}{\epsilon\delta^2}\sum\limits_{i>l}\frac{\xi_i^2}{\lambda_1^2}\\
			\ge&(1+\vartheta-\delta\vartheta-\epsilon)\frac{\xi_1^2}{\lambda_1^2}-\delta_0 \sum\limits_{i>l}\frac{S_k^{ii}\xi_i^2}{\lambda_1 S_k}.
		\end{align*}
		The concavity inequality is now proved.
		
	\end{proof}

\subsection{A concavity inequality about $S_n$}
\begin{lemma}\label{lem3.1}
	Let $\la=(\la_1,\la_2,\cdots,\la_n)\in\mathbb{R}^{n}$, then we have\\
	(i)$$S_n^{jj}(-S_n^{jj}+2\la_1S_n^{11,jj}+S_n^{11})=\la_1^2S_{n-2}^2(\la|1j)+S_nS_{n-2}(\la|1j)$$
	(ii)\begin{align*}
		&\la_1(S_n^{pp}S_n^{11,qq}+S_n^{qq}S_n^{11,pp}-S_n^{11}S_n^{pp,qq})-S_n^{pp}S_n^{qq}\\
		=&\la_1^2\sigma_{n-3}^2(\la|1pq)-S_n(\la)\sigma_{n-3}(\la|1pq)
	\end{align*}
	(iii)$$-\la_1^2S_n^{11,pp}S_n^{11,qq}+\la_1S_n^{11}S_n^{pp,qq}=-\la_1^2\sigma_{n-3}^2(\la|1pq)+\la_1	S_{n-1}(\la|1)\sigma_{n-3}(\la|1pq)$$
\end{lemma}

\begin{lemma}\label{lem3.2}
	If $\lambda\in\widetilde{
		\Gamma}_n$, $\lambda_1\geq\cdots\geq\lambda_n$ then this quadratic form
	$$\sum\limits_{j>1}S_{n-2}(\la|1j)\xi_j^2-\sum\limits_{p,q>1;p\neq q}\sigma_{n-3}(\la|1pq)\xi_p\xi_q\ge0.$$
\end{lemma}
\begin{proof}
	A  straightforward calculation shows
	\begin{align}\label{y3.1}
		&	\sum\limits_{j>1}S_{n-2}(\la|1j)\xi_j^2-\sum\limits_{p,q>1;p\neq q}\sigma_{n-3}(\la|1pq)\xi_p\xi_q\notag\\
		=&\sum\limits_{1<j<n}S_{n-2}(\la|1j)\xi_j^2-\sum\limits_{1<p,q<n;p\neq q}\sigma_{n-3}(\la|1pq)\xi_p\xi_q\notag\\
		&+S_{n-2}(\la|1n)\xi_n^2-2\sum\limits_{1<j<n}\sigma_{n-3}(\la|1jn)\xi_j\xi_n.
	\end{align}	
	For $1<j<n; 1<p,q<n; p\ne q$, we have
	\begin{align}\label{y3.2}
		&S_{n-2}(\la|1j)S_{n-2}(\la|1n)\notag\\
		=&\big[\la_nS_{n-3}(\la|1jn)+\sigma_{n-3}(\la|1jn)\big]S_{n-2}(\la|1n)\notag\\
		=&S_{n-3}(\la|1jn)\big[S_{n-1}(\la|1)-\sigma_{n-2}(\la|1n)\big]+\sigma_{n-3}(\la|1jn)S_{n-2}(\la|1n)\notag\\
		=&S_{n-3}(\la|1jn)S_{n-1}(\la|1)-\la_jS_{n-3}(\la|1jn)\sigma_{n-3}(\la|1jn)\notag\\
		&+\sigma_{n-3}(\la|1jn)\big[\la_jS_{n-3}(\la|1jn)+\sigma_{n-3}(\la|1jn)\big]\notag\\
		=&S_{n-3}(\la|1jn)S_{n-1}(\la|1)+\sigma_{n-3}^2(\la|1jn),
	\end{align}
	\begin{align}\label{y3.3}
		\sigma_{n-3}(\la|1pq)S_{n-2}(\la|1n)&=\sigma_{n-4}(\la|1pqn)\la_nS_{n-2}(\la|1n)\notag\\
		&=\sigma_{n-4}(\la|1pqn)\big[S_{n-1}(\la|1)-\sigma_{n-2}(\la|1n)\big]\notag\\
		&=\sigma_{n-4}(\la|1pqn)S_{n-1}(\la|1)-\sigma_{n-4}(\la|1pqn)\sigma_{n-2}(\la|1n),
	\end{align}
	\begin{align}\label{y3.4}
		&S_{n-2}(\la|1n)\bigg[S_{n-2}(\la|1n)\xi_n^2-2\sum\limits_{1<j<n}\sigma_{n-3}(\la|1jn)\xi_j\xi_n\bigg]\notag\\
		=&S_{n-2}^2(\la|1n)\xi_n^2-2\left(\sum\limits_{1<j<n}\sigma_{n-3}(\la|1jn)\xi_j\right)S_{n-2}(\la|1n)\xi_n\notag\\
		\ge&-\left(\sum\limits_{1<j<n}\sigma_{n-3}(\la|1jn)\xi_j\right)^2\notag\\
		=&-\sum\limits_{1<j<n}\sigma_{n-3}^2(\la|1jn)\xi_j^2-\sum\limits_{1<p,q<n;p\neq q}\sigma_{n-3}(\la|1pn)\sigma_{n-3}(\la|1qn)\xi_p\xi_q\notag\\
		=&-\sum\limits_{1<j<n}\sigma_{n-3}^2(\la|1jn)\xi_j^2-\sum\limits_{1<p,q<n;p\neq q}\sigma_{n-2}(\la|1n)\sigma_{n-4}(\la|1pqn)\xi_p\xi_q.
	\end{align}
	Plugging \eqref{y3.2} \eqref{y3.3} \eqref{y3.4} into \eqref{y3.1}, we have
	\begin{align*}
		&S_{n-2}(\la|1n)\bigg[\sum\limits_{j>1}S_{n-2}(\la|1j)\xi_j^2-\sum\limits_{p,q>1;p\neq q}\sigma_{n-3}(\la|1pq)\xi_p\xi_q\bigg]\\
		\ge&\sum\limits_{1<j<n}\big[S_{n-3}(\la|1jn)S_{n-1}(\la|1)+\sigma_{n-3}^2(\la|1jn)\big]\xi_j^2\\
		&-\sum\limits_{1<p,q<n;p\neq q}\big[\sigma_{n-4}(\la|1pqn)S_{n-1}(\la|1)-\sigma_{n-4}(\la|1pqn)\sigma_{n-2}(\la|1n)\big]\xi_p\xi_q\\
		&-\sum\limits_{1<j<n}\sigma_{n-3}^2(\la|1jn)\xi_j^2-\sum\limits_{1<p,q<n;p\neq q}\sigma_{n-2}(\la|1n)\sigma_{n-4}(\la|1pqn)\xi_p\xi_q\\
		=&S_{n-1}(\la|1)\biggl[\sum\limits_{1<j<n}S_{n-3}(\la|1jn)\xi_j^2-\sum\limits_{1<p,q<n;p\neq q}\sigma_{n-4}(\la|1pqn)\xi_p\xi_q\bigg]\\
		\ge&S_{n-1}(\la|1)\biggl[\sum\limits_{1<j<n}\sigma_{n-4}(\la|1jn)\xi_j^2-\sum\limits_{1<p,q<n;p\neq q}\sigma_{n-4}(\la|1pqn)\xi_p\xi_q\bigg]\\
		\ge&0.
	\end{align*}
	Here, in the last step, we have used this fact. If $(a_{ij})_{m\times m}$ is a real symmetric matrix of order $m$, and $a_{ij}>0$, $\forall i, a_{ii}=\sum\limits_{j\ne i}a_{ij}$, then
	\begin{align*}
		\sum\limits_{i\ne j}a_{ij}\xi_i\xi_j\le& \frac{1}{2}\sum\limits_{i\ne j}a_{ij}(\xi_i^2+\xi_j^2)\\
		=&\frac{1}{2}\sum\limits_{i\ne 1}a_{i1}(\xi_i^2+\xi_1^2)+\cdots+\frac{1}{2}\sum\limits_{i\ne m}a_{im}(\xi_i^2+\xi_m^2)\\
		=&\frac{1}{2}\sum\limits_{i\ne 1}a_{i1}\xi_i^2+\cdots+\frac{1}{2}\sum\limits_{i\ne m}a_{im}\xi_i^2+\frac{\xi_1^2}{2}a_{11}+\cdots+\frac{\xi_m^2}{2}a_{mm}\\
		=&\frac{1}{2}\left(\sum\limits_ia_{i1}\xi_i^2-a_{11}\xi_1^2\right)+\cdots+\frac{1}{2}\left(\sum\limits_ia_{im}\xi_i^2-a_{mm}\xi_m^2\right)\\
		&+\frac{\xi_1^2}{2}a_{11}+\cdots+\frac{\xi_m^2}{2}a_{mm}\\
		=&\frac{1}{2}\sum\limits_i(a_{i1}+\cdots+a_{im})\xi_i^2\\
		=&\sum\limits_ia_{ii}\xi_i^2.
	\end{align*}

\end{proof}

\begin{rema}\label{rema3.1}
	By Schur product theorem, we have
	\begin{equation*}
		\sum\limits_{j>1}S_{n-2}^2(\la|1j)\xi_j^2+\sum\limits_{p,q>1;p\neq q}\sigma_{n-3}^2(\la|1pq)\xi_p\xi_q\ge0.
	\end{equation*}
\end{rema}

\begin{lemma}\label{lem3.3}
	If $\lambda\in\widetilde{
		\Gamma}_n$, then $\forall \epsilon>0$, $\exists K(\epsilon)$, such that
	\begin{equation}\label{y1}
		\lambda_1\left(K\left(\sum\limits_jS_n^{jj}(\lambda)\xi_j\right)^2-S_n^{pp,qq}(\lambda)\xi_p\xi_q\right)-S_n^{11}(\lambda)\xi_1^2+(1+\epsilon)\sum\limits_{j> 1}S_n^{jj}\xi_j^2\ge0.
	\end{equation}
\end{lemma}

\begin{proof}
	A straightforward calculation shows
	\begin{align}\label{y2}
		&\lambda_1\left(K\left(\sum\limits_jS_n^{jj}(\lambda)\xi_j\right)^2-S_n^{pp,qq}(\lambda)\xi_p\xi_q\right)-S_n^{11}(\lambda)\xi_1^2+(1+\epsilon)\sum\limits_{j> 1}S_n^{jj}\xi_j^2\notag\\
		=&\la_1K\left(\sum\limits_{j> 1}S_n^{jj}(\lambda)\xi_j\right)^2+2\la_1\xi_1\left[\sum\limits_{j>1}(KS_n^{11}S_n^{jj}-S_n^{11,jj})\xi_j\right]\notag\\
		&+\left[\la_1K(S_n^{11})^2-S_n^{11}\right]\xi_1^2+(1+\epsilon)\sum\limits_{j> 1}S_n^{jj}\xi_j^2-\la_1\sum\limits_{p>1,q>1}S_n^{pp,qq}(\lambda)\xi_p\xi_q\notag\\
		\ge&\la_1K\left(\sum\limits_{j> 1}S_n^{jj}(\lambda)\xi_j\right)^2-\frac{\la_1^2\bigg[\sum\limits_{j>1}(KS_n^{11}S_n^{jj}-S_n^{11,jj})\xi_j\bigg]^2}{\la_1K(S_n^{11})^2-S_n^{11}}\notag\\
		&+(1+\epsilon)\sum\limits_{j> 1}S_n^{jj}\xi_j^2-\la_1\sum\limits_{p>1,q>1}S_n^{pp,qq}(\lambda)\xi_p\xi_q\notag\\
		=&\sum\limits_{j>1}\left[\la_1K(S_n^{jj})^2-\frac{\la_1^2(KS_n^{11}S_n^{jj}-S_n^{11,jj})^2}{\la_1K(S_n^{11})^2-S_n^{11}}+(1+\epsilon)S_n^{jj}\right]\xi_j^2+\notag\\
		&\sum\limits_{p>1,q>1,p\ne q}\bigg[\la_1KS_n^{pp}S_n^{qq}-\frac{\la_1^2(KS_n^{11}S_n^{pp}-S_n^{11,pp})(KS_n^{11}S_n^{qq}-S_n^{11,qq})}{\la_1K(S_n^{11})^2-S_n^{11}}\notag\\
		&-\la_1S_n^{pp,qq}\bigg]\xi_p\xi_q.
	\end{align}
	where, in the second step, we have used the Cauchy inequality with $\epsilon$, that is
	$$2\la_1\xi_1\left[\sum\limits_{j>1}(KS_n^{11}S_n^{jj}-S_n^{11,jj})\xi_j\right]\ge -\frac{1}{\epsilon}\la_1^2\left[\sum\limits_{j>1}(KS_n^{11}S_n^{jj}-S_n^{11,jj})\xi_j\right]^2-\epsilon\xi_1^2,$$
	and let $\epsilon=\la_1K(S_n^{11})^2-S_n^{11}.$\\
	Choosing $K$ is sufficiently large, we have
	$$K\la_1S_n^{11}-1\ge K\theta S_n-1>0,$$
	Thus, we can multiple the term $S_n^{11}(K\la_1S_n^{11}-1)$ in \eqref{y2} and combining Lemma \eqref{lem3.1}. Then, we obtain
	\begin{align*}
		&(\la_1K(S_n^{11})^2-S_n^{11})\\
		&\times\left[\lambda_1\left(K\left(\sum\limits_jS_n^{jj}(\lambda)\xi_j\right)^2-S_n^{pp,qq}(\lambda)\xi_p\xi_q\right)-S_n^{11}(\lambda)\xi_1^2+(1+\epsilon)\sum\limits_{j> 1}S_n^{jj}\xi_j^2\right]\\
		\ge&\sum\limits_{j>1}\bigg[\la_1KS_n^{11}S_n^{jj}(-S_n^{jj}+2\la_1S_n^{11,jj}+(1+\epsilon)S_n^{11})-\la_1^2(S_n^{11,jj})^2-(1+\epsilon)S_n^{11}S_n^{jj}\bigg]\xi_j^2\\
		&+\sum\limits_{p>1,q>1,p\ne q}\bigg[\la_1KS_n^{11}(\la_1(S_n^{pp}S_n^{11,qq}+S_n^{qq}S_n^{11,pp}-S_n^{11}S_n^{pp,qq})-S_n^{pp}S_n^{qq})\\
		&-\la_1^2S_n^{11,pp}S_n^{11,qq}+\la_1S_n^{11}S_n^{pp,qq}\bigg]\xi_p\xi_q\\
		\ge&\sum\limits_{j>1}\bigg[\la_1KS_n^{11}(\la_1^2S_{n-2}^2(\la|1j)+S_nS_{n-2}(\la|1j)+\epsilon S_n^{jj}S_n^{11})\\
		&-\la_1^2(S_n^{11,jj})^2-(1+\epsilon)S_n^{11}S_n^{jj}\bigg]\xi_j^2+\\
		&+\sum\limits_{p>1,q>1,p\ne q}\bigg[\la_1KS_n^{11}(\la_1^2\sigma_{n-3}^2(\la|1pq)-S_n(\la)\sigma_{n-3}(\la|1pq))\\
		&-\la_1^2\sigma_{n-3}^2(\la|1pq)+\la_1	S_{n-1}(\la|1)\sigma_{n-3}(\la|1pq)\bigg]\xi_p\xi_q\\
		=&\sum\limits_{j>1}\bigg[(\la_1KS_n^{11}-1)\la_1^2S_{n-2}^2(\la|1j)+\la_1KS_n^{11}S_nS_{n-2}(\la|1j)\\
		&+\la_1KS_n^{11}\epsilon S_n^{jj}S_n^{11}-(1+\epsilon)S_n^{11}S_n^{jj}\bigg]\xi_j^2+\\
		&+\sum\limits_{p>1,q>1,p\ne q}\bigg[(\la_1KS_n^{11}-1)\la_1^2\sigma_{n-3}^2(\la|1pq)-\la_1KS_n^{11}S_n(\la)\sigma_{n-3}(\la|1pq)\\
		&+\la_1	S_{n-1}(\la|1)\sigma_{n-3}(\la|1pq)\bigg]\xi_p\xi_q\\
		\ge&\sum\limits_{j>1}\bigg[(\la_1KS_n^{11}-1)\la_1^2S_{n-2}^2(\la|1j)+\la_1KS_n^{11}S_nS_{n-2}(\la|1j)\\
		&\qquad+(K\theta S_n\epsilon-(1+\epsilon))S_n^{11}S_n^{jj}\bigg]\xi_j^2\\
		&+\sum\limits_{p,q>1;p\ne q}\bigg[(\la_1KS_n^{11}-1)\la_1^2\sigma_{n-3}^2(\la|1pq)\\
		&\qquad-\la_1KS_n^{11}(S_n(\la)-\frac{1}{K})\sigma_{n-3}(\la|1pq)\bigg]\xi_p\xi_q.
	\end{align*}
	Choosing $K$ is sufficiently large, such that
	$$K\theta S_n\epsilon-(1+\epsilon)>0,$$
	hence
	\begin{align*}
		&(\la_1K(S_n^{11})^2-S_n^{11})\\
		&\times\left[\lambda_1\left(K\left(\sum\limits_jS_n^{jj}(\lambda)\xi_j\right)^2-S_n^{pp,qq}(\lambda)\xi_p\xi_q\right)-S_n^{11}(\lambda)\xi_1^2+(1+\epsilon)\sum\limits_{j> 1}S_n^{jj}\xi_j^2\right]\\
		\ge&\sum\limits_{j>1}\bigg[(\la_1KS_n^{11}-1)\la_1^2S_{n-2}^2(\la|1j)+\la_1KS_n^{11}(S_n-\frac{1}{K})S_{n-2}(\la|1j)\bigg]\xi_j^2+\\
		+&\sum\limits_{p,q>1;p\ne q}\bigg[(\la_1KS_n^{11}-1)\la_1^2\sigma_{n-3}^2(\la|1pq)-\la_1KS_n^{11}(S_n(\la)-\frac{1}{K})\sigma_{n-3}(\la|1pq)\bigg]\xi_p\xi_q\\
		=&(\la_1KS_n^{11}-1)\la_1^2\bigg[\sum\limits_{j>1}S_{n-2}^2(\la|1j)\xi_j^2+\sum\limits_{p,q>1;p\ne q}\sigma_{n-3}^2(\la|1pq)\xi_p\xi_q\bigg]\\
		&+\la_1KS_n^{11}(S_n(\la)-\frac{1}{K})\bigg[\sum\limits_{j>1}S_{n-2}(\la|1j)\xi_j^2-\sum\limits_{p,q>1;p\ne q}\sigma_{n-3}(\la|1pq)\xi_p\xi_q\bigg]\\
		\ge&0.
	\end{align*}
	Here, the last step holds by Lemma \eqref{lem3.2} and Remark \eqref{rema3.1}.
	
\end{proof}

	\section{Pogorelov type $C^2$ estimates for $S_k$ equations}
	In this section, we will prove Theorem 1.1(a). We consider the following test function.
	$$\widetilde{P}(x)=\ln{\lambda_{max}}+\beta\ln(-u)+\frac{a}{2}|Du|^2+\frac{A}{2}|x|^2,$$
	where $\la_{max}(x)$ is the biggest eigenvalue of the Hessian matrix, $\beta,a$ and $A$ are constants which will be determined later. Following the analysis of $\ln{\lambda_{max}}$ in \cite{JC}. Suppose $\widetilde{P}$ attain its maximum value in $\Omega$ at $x_0$. Rotating the coordinates, we assume that the matrix $(u_{ij})$ is diagonal at $x_0$. We assume $\la_1(x_0)$ has multiplicity m, then
	$$u_{ij}=u_{ii}\delta_{ij},\lambda_i=u_{ii},\lambda_1=\lambda_2=\cdots=\la_m>\la_{m+1}\ge\cdots\ge\la_n.$$
	We now apply a perturbation argument. Let $B$ be a matrix satisfying the following conditions:
	$$B_{ij}=\delta_{ij}(1-\delta_{1i}),B_{ij,p}=B_{11,ii}=0.$$
	Define the matrix by $\widetilde{u}_{ij}=u_{ij}-B_{ij}$ and denote its eigenvalues by $\widetilde{\lambda}=(\widetilde{\lambda}_1,\widetilde{\lambda}_2,\cdots,\widetilde{\lambda}_n).$
	Hence 
	$$\widetilde{\lambda_1}=\la_1,\widetilde{\lambda_i}=\la_i-1, i=2,\cdots,n.$$
	It follows that $\widetilde{\lambda}_1>\widetilde{\lambda}_2$, which implies that $\widetilde{\lambda}_1$ is smooth at $x_0$. We consider the perturbed quantity $P(x)$ defined by
	$$P(x)=\ln\widetilde{\la}_1+\beta\ln(-u)+\frac{a}{2}|Du|^2+\frac{A}{2}|x|^2.$$
	Differentiating $P(x)$ twice and use Lemma 2.5, at $x_0$, we have
	\begin{equation}\label{4.1}
		\frac{u_{11i}}{\lambda_1}+\frac{\beta u_i}{u}+au_iu_{ii}+Ax_i=0,
	\end{equation}
	\begin{equation}\label{4.2}
		\frac{\beta u_{ii}}{u}-\frac{\beta u_i^2}{u^2}+\frac{u_{11ii}}{\lambda_1}+2\sum\limits_{p>1}\frac{u_{1pi}^2}{\lambda_1(\lambda_1-\widetilde{\lambda_p})}-\frac{u_{11i}^2}{\lambda_1^2}+a\sum\limits_ju_ju_{jii}+au_{ii}^2+A\le 0.
	\end{equation}
	Multiplying both sides of \eqref{4.2} by $S_k^{ii}$ and summing them.
	\begin{align}\label{4.3}
		0\ge&	\frac{\beta S_k^{ii}u_{ii}}{u}-\frac{\beta S_k^{ii}u_i^2}{u^2}+\frac{S_k^{ii}u_{11ii}}{\lambda_1}+2\sum\limits_{p>1}\frac{S_k^{ii}u_{1pi}^2}{\lambda_1(\lambda_1-\widetilde{\lambda_p})}\notag\\
		&-\frac{S_k^{ii}u_{11i}^2}{\lambda_1^2}+a\sum\limits_jS_k^{ii}u_ju_{jii}+a\sum\limits_iS_k^{ii}u_{ii}^2+A\sum\limits_iS_k^{ii}.
	\end{align}
	At $x_0$, differentiating equation (1.1) twice, we have
	\begin{equation}\label{4.4}
		\sum\limits_iS_k^{ii}u_{iij}=f_j=f_{x_j}+f_uu_j+f_{u_j}u_{jj}\le C(1+\lambda_1),
	\end{equation}
	\begin{equation}\label{4.5}
		\sum\limits_{ij,rs}S_k^{ij,rs}u_{ij1}u_{rs1}+\sum\limits_iS_k^{ii}u_{ii11}=f_{11}\ge -C(1+\lambda_1+\lambda_1^2)+\sum\limits_if_{u_i}u_{11i}.
	\end{equation}
	Using Lemma \eqref{lem2.5}, we have
	\begin{align*}
		\frac{S_k^{ii}u_{ii11}}{\lambda_1}&\ge \frac{1}{\lambda_1}\left[-\sum\limits_{ij,rs}S_k^{ij,rs}u_{ij1}u_{rs1}-C(1+\lambda_1+\lambda_1^2)+\sum\limits_if_{u_i}u_{11i}\right]\notag\\
		&=\frac{1}{\lambda_1}\left[-\sum\limits_{p\ne q}S_k^{pp,qq}u_{pp1}u_{qq1}+\sum\limits_{p\ne q}S_k^{pp,qq}u_{pq1}^2-C(1+\lambda_1+\lambda_1^2)+\sum\limits_if_{u_i}u_{11i}\right]\notag\\
		&\ge -\frac{1}{\lambda_1}\sum\limits_{p\ne q}S_k^{pp,qq}u_{pp1}u_{qq1}+\frac{2}{\lambda_1}\sum\limits_{i>m}S_k^{11,ii}u_{11i}^2-C(1+\lambda_1)+\frac{1}{\lambda_1}\sum\limits_if_{u_i}u_{11i}\notag\\
		&= -\frac{1}{\lambda_1}\sum\limits_{p\ne q}S_k^{pp,qq}u_{pp1}u_{qq1}+\sum\limits_{i>m}\frac{2(S_k^{ii}-S_k^{11})u_{11i}^2}{\lambda_1(\lambda_1-\lambda_i)}-C(1+\lambda_1)+\frac{1}{\lambda_1}\sum\limits_if_{u_i}u_{11i}.
	\end{align*}
	By \eqref{4.4} , we obtain
	\begin{align*}
		a\sum\limits_i\sum\limits_jS_k^{ii}u_ju_{iij}\ge -C(A+a+\frac{1}{-u})-\frac{1}{\lambda_1}\sum\limits_jf_{u_j}u_{11j}.
	\end{align*}
	By \eqref{4.1}, we have
	\begin{align*}
		\frac{\beta S_k^{ii}u_i^2}{u^2}&\le \frac{3S_k^{ii}u_{11i}^2}{\beta\lambda_1^2}+\frac{3a^2}{\beta}S_k^{ii}u_i^2u_{ii}^2+\frac{3A^2}{\beta}S_k^{ii}x_i^2\\
		&\le\frac{3S_k^{ii}u_{11i}^2}{\beta\lambda_1^2}+\frac{C_1a^2}{\beta}S_k^{ii}u_{ii}^2+\frac{C_2A^2}{\beta}S_k^{ii}.
	\end{align*}
	Here $C_1=3\sup\limits_{\Omega}|Du|^2$, $C_2=3(diam\Omega)^2$.\\
	By Lemma \eqref{lem2.1} (v), we have
	$$\frac{\beta S_k^{ii}u_{ii}}{u}=\frac{\beta}{u}(kS_k(\lambda)-\sigma_{k-1}(\lambda))\ge\frac{\beta}{u}kS_k(\lambda)\ge \frac{C\beta}{u}.$$
	Plugging the above three inequalities into \eqref{4.3} and choosing 
	$\beta>\max\left\{12,2C_1a,2C_2A\right\}$ , we have
	\begin{align}\label{4.6}
		0\ge&\frac{\beta S_k^{ii}u_{ii}}{u}-\frac{\beta S_k^{ii}u_i^2}{u^2}-\frac{1}{\lambda_1}\sum\limits_{p\ne q}S_k^{pp,qq}u_{pp1}u_{qq1}+\sum\limits_{i>m}\frac{2(S_k^{ii}-S_k^{11})u_{11i}^2}{\lambda_1(\lambda_1-\lambda_i)}\notag\\
		&-C(1+\lambda_1)-C(A+a+\frac{1}{-u})+2\sum\limits_i\sum\limits_{p>1}\frac{S_k^{ii}u_{1pi}^2}{\lambda_1(\lambda_1-\widetilde{\lambda_p})}\notag\\
		&-\frac{S_k^{ii}u_{11i}^2}{\lambda_1^2}+a\sum\limits_iS_k^{ii}u_{ii}^2+A\sum\limits_iS_k^{ii}\notag\\
		\ge&\frac{C\beta}{u}-\frac{\beta S_k^{11}u_1^2}{u^2}-\frac{1}{\lambda_1}\sum\limits_{p\ne q}S_k^{pp,qq}u_{pp1}u_{qq1}+\sum\limits_{i>m}\frac{2(S_k^{ii}-S_k^{11})u_{11i}^2}{\lambda_1(\lambda_1-\lambda_i)}\notag\\
		&-C(A+a+\frac{1}{-u}+\lambda_1)+2\sum\limits_{p>1}\frac{S_k^{11}u_{11p}^2}{\lambda_1(\lambda_1-\widetilde{\lambda_p})}+2\sum\limits_{p>1}\frac{S_k^{pp}u_{1pp}^2}{\lambda_1(\lambda_1-\widetilde{\lambda_p})}-\frac{S_k^{11}u_{111}^2}{\lambda_1^2}\notag\\
		&-\left(1+\frac{3}{\beta}\right)\sum\limits_{i>1}\frac{S_k^{ii}u_{11i}^2}{\lambda_1^2}+\left(a-\frac{C_1a^2}{\beta}\right)\sum\limits_iS_k^{ii}u_{ii}^2+\left(A-\frac{C_2A^2}{\beta}\right)\sum\limits_iS_k^{ii}\notag\\
		\ge&\frac{C\beta}{u}-\frac{\beta S_k^{11}u_1^2}{u^2}-\frac{1}{\lambda_1}\sum\limits_{p\ne q}S_k^{pp,qq}u_{pp1}u_{qq1}+\sum\limits_{i>m}\frac{2(S_k^{ii}-S_k^{11})u_{11i}^2}{\lambda_1(\lambda_1-\lambda_i)}\notag\\
		&-C(A+a+\frac{1}{-u}+\lambda_1)+2\sum\limits_{p>1}\frac{S_k^{11}u_{11p}^2}{\lambda_1(\lambda_1-\widetilde{\lambda_p})}+2\sum\limits_{p>1}\frac{S_k^{pp}u_{1pp}^2}{\lambda_1(\lambda_1-\widetilde{\lambda_p})}-\frac{S_k^{11}u_{111}^2}{\lambda_1^2}\notag\\
		&-\frac{5}{4}\sum\limits_{i>1}\frac{S_k^{ii}u_{11i}^2}{\lambda_1^2}+\frac{a}{2}\sum\limits_iS_k^{ii}u_{ii}^2+\frac{A}{2}\sum\limits_iS_k^{ii}.
	\end{align}
	
	\begin{lemma}\label{lem4.1}
		$$\sum\limits_{i>m}\frac{2(S_k^{ii}-S_k^{11})u_{11i}^2}{\lambda_1(\lambda_1-\lambda_i)}+2\sum\limits_{p>1}\frac{S_k^{11}u_{11p}^2}{\lambda_1(\lambda_1-\widetilde{\lambda_p})}-\frac{5}{4}\sum\limits_{i>1}\frac{S_k^{ii}u_{11i}^2}{\lambda_1^2}\ge 0,$$
		by choosing $\lambda_1$ sufficiently large.
	\end{lemma}
	\begin{proof}
		\begin{align*}
			&\sum\limits_{i>m}\frac{2(S_k^{ii}-S_k^{11})u_{11i}^2}{\lambda_1(\lambda_1-\lambda_i)}+2\sum\limits_{p>1}\frac{S_k^{11}u_{11p}^2}{\lambda_1(\lambda_1-\widetilde{\lambda_p})}-\frac{5}{4}\sum\limits_{i>1}\frac{S_k^{ii}u_{11i}^2}{\lambda_1^2}\notag\\
			=&\sum\limits_{p>m}\frac{2(S_k^{pp}-S_k^{11})u_{11p}^2}{\lambda_1(\lambda_1-\lambda_p)}+2\sum\limits_{1<p\le m}\frac{S_k^{11}u_{11p}^2}{\lambda_1(\lambda_1-\widetilde{\lambda_p})}+2\sum\limits_{p>m}\frac{S_k^{11}u_{11p}^2}{\lambda_1(\lambda_1-\widetilde{\lambda_p})}\notag\\
			&-\frac{5}{4}\sum\limits_{1<p\le m}\frac{S_k^{pp}u_{11p}^2}{\lambda_1^2}-\frac{5}{4}\sum\limits_{p>m}\frac{S_k^{pp}u_{11p}^2}{\lambda_1^2}\notag\\
			=&\sum\limits_{p>m}\frac{S_k^{11}u_{11p}^2}{\lambda_1}\left(\frac{2}{\lambda_1-\widetilde{\lambda_{p}}}-\frac{2}{\lambda_1-\lambda_p}\right)+\sum\limits_{p>m}\frac{S_k^{pp}u_{11p}^2}{\lambda_1}\left(\frac{2}{\lambda_1-\lambda_p}-\frac{\frac{5}{4}}{\lambda_1}\right)\notag\\
			&+2\sum\limits_{1<p\le m}\frac{S_k^{11}u_{11p}^2}{\lambda_1(\lambda_1-\widetilde{\lambda_p})}-\frac{5}{4}\sum\limits_{1<p\le m}\frac{S_k^{11}u_{11p}^2}{\lambda_1^2}\notag\\
			=&\sum\limits_{p>m}\frac{S_k^{11}u_{11p}^2}{\lambda_1}\frac{-2}{(\lambda_1-\lambda_p)(\lambda_1-\lambda_p+1)}+\sum\limits_{p>m}\frac{S_k^{pp}u_{11p}^2}{\lambda_1}\frac{\frac{3}{4}\lambda_1+\frac{5}{4}\lambda_p}{\lambda_1(\lambda_1-\lambda_p)}\notag\\
			&+\sum\limits_{1<p\le m}\frac{S_k^{11}u_{11p}^2}{\lambda_1}\left(\frac{2}{\lambda_1-\widetilde{\lambda_p}}-\frac{\frac{5}{4}}{\lambda_1}\right)\notag\\
			\ge&\sum\limits_{p>m}\frac{S_k^{11}u_{11p}^2}{\lambda_1}\frac{-2}{(\lambda_1-\lambda_p)(\lambda_1-\lambda_p+1)}+\sum\limits_{p>m}\frac{S_k^{11}u_{11p}^2}{\lambda_1}\frac{\frac{3}{4}\lambda_1+\frac{5}{4}\lambda_p}{\lambda_1(\lambda_1-\lambda_p)}\notag\\
			&+\sum\limits_{1<p\le m}\frac{S_k^{11}u_{11p}^2}{\lambda_1}\frac{\frac{3}{4}\lambda_1+\frac{5}{4}\widetilde{\lambda_p}}{\lambda_1(\lambda_1-\widetilde{\lambda_p})}\notag\\
			\ge&\sum\limits_{p>m}\frac{S_k^{11}u_{11p}^2}{\lambda_1}\left(\frac{\frac{3}{4}\lambda_1+\frac{5}{4}\lambda_p}{\lambda_1(\lambda_1-\lambda_p)}-\frac{2}{(\lambda_1-\lambda_p)(\lambda_1-\lambda_p+1)}\right)\notag\\
			=&\sum\limits_{p>m}\frac{S_k^{11}u_{11p}^2}{\lambda_1}\frac{\frac{3}{5}\lambda_1+\lambda_p-1}{\lambda_1(\lambda_1-\lambda_p+1)}\notag\\
			\ge&0.
		\end{align*}
		
	\end{proof}
	Plugging Lemma \ref{lem4.1} into \eqref{4.6}, we get
	\begin{align}\label{4.7}
		0\ge&\frac{C\beta}{u}-\frac{\beta S_k^{11}u_1^2}{u^2}-\frac{1}{\lambda_1}\sum\limits_{p\ne q}S_k^{pp,qq}u_{pp1}u_{qq1}+2\sum\limits_{p>1}\frac{S_k^{pp}u_{1pp}^2}{\lambda_1(\lambda_1-\widetilde{\lambda_p})}\notag\\
		&-\frac{S_k^{11}u_{111}^2}{\lambda_1^2}+\frac{a}{2}\sum\limits_iS_k^{ii}u_{ii}^2+\frac{A}{2}\sum\limits_iS_k^{ii}-C(A+a+\frac{1}{-u}+\lambda_1).
	\end{align}
	\begin{lemma}\label{lem4.2}
		$\forall \epsilon_0\in (0,\frac{1}{2})$, if we choose $\delta\in(0,\frac{\epsilon_0}{4}),\epsilon\in(0,\frac{\epsilon_0\vartheta}{4}),\delta_0=\frac{2k}{n+k}, m\le l<k$, there exist constants $\delta^{'}$ depending only on $\epsilon,\delta,\delta_0,n,k,l$, such that if $\lambda_l \ge \delta\lambda_1,\lambda_{l+1} \le \delta^{'}\lambda_1$, then
		$$(-u)^{\beta}\lambda_1\le C,$$
		where $C$ depends on $\epsilon,\delta,\delta_0,n,k,l,|u|_{C^1},\inf f,|f|_{C^2}.$
	\end{lemma}
	\begin{proof}
		Using the concavity inequality in section 3.2, we have
		\begin{align*}
			&-\sum\limits_{p\ne q}\frac{S_k^{pp,qq}u_{pp1} u_{qq1}}{S_k}+(1-\vartheta+\frac{\vartheta}{\delta})\frac{\left(\sum\limits_{i}S_k^{ii}u_{ii1} \right)^2}{S_k^2}\notag\\
			\ge& (1+\vartheta-\delta\vartheta-\epsilon)\frac{u_{111}^2}{\lambda_1^2}-\delta_0 \sum\limits_{i>l}\frac{S_k^{ii}u_{ii1}^2}{\lambda_1 S_k},
		\end{align*}
		it follows that
		\begin{align}\label{4.8}
			&-\frac{1}{\lambda_1}\sum\limits_{p\ne q}S_k^{pp,qq}u_{pp1}u_{qq1}+2\sum\limits_{p>1}\frac{S_k^{pp}u_{1pp}^2}{\lambda_1(\lambda_1-\widetilde{\lambda_p})}\notag\\
			\ge&-(1-\vartheta+\frac{\vartheta}{\delta})\frac{\left(\sum\limits_{i}S_k^{ii}u_{ii1} \right)^2}{\lambda_1S_k}+(1+\vartheta-\delta\vartheta-\epsilon)\frac{S_ku_{111}^2}{\lambda_1^3}\notag\\
			&-\delta_0 \sum\limits_{i>l}\frac{S_k^{ii}u_{ii1}^2}{\lambda_1^2}+2\sum\limits_{p>1}\frac{S_k^{pp}u_{1pp}^2}{\lambda_1(\lambda_1-\widetilde{\lambda_p})}.
		\end{align}
		By \eqref{4.4}, we have
		\begin{align}\label{4.9}
			\left(\sum\limits_iS_k^{ii}u_{ii1}\right)^2=&\left([S_k(D^2u)]_{x_1}\right)^2=(f_{x_1}+f_uu_1+f_{u_1}u_{11})^2\notag\\
			\le&3(f_{x_1}^2+f_u^2u_1^2+f_{u_1}^2u_{11}^2)\le C+Cu_{11}^2\le C\lambda_1^2.
		\end{align}
		Note that
		\begin{equation}\label{4.10}
			\vartheta-\delta\vartheta-\epsilon\ge \vartheta-\frac{\epsilon_0\vartheta}{4}-\frac{\epsilon_0\vartheta}{4}=\vartheta(1-\frac{\epsilon_0}{2})\ge 0,
		\end{equation}
		\begin{align}\label{4.11}
			&-\delta_0 \sum\limits_{i>l}\frac{S_k^{ii}u_{ii1}^2}{\lambda_1^2}+2\sum\limits_{p>1}\frac{S_k^{pp}u_{1pp}^2}{\lambda_1(\lambda_1-\widetilde{\lambda_p})}\notag\\
			\ge&-\delta_0 \sum\limits_{p>l}\frac{S_k^{pp}u_{1pp}^2}{\lambda_1^2}+2\sum\limits_{p>l}\frac{S_k^{pp}u_{1pp}^2}{\lambda_1(\lambda_1-\widetilde{\lambda_p})}\notag\\
			=&\sum\limits_{p>l}\frac{S_k^{pp}u_{1pp}^2}{\lambda_1}\left(\frac{2}{\lambda_1-\widetilde{\lambda_p}}-\frac{\delta_0}{\lambda_1}\right)\notag\\
			=&\sum\limits_{p>l}\frac{S_k^{pp}u_{1pp}^2}{\lambda_1}\frac{(2-\delta_0)\lambda_1+\delta_0\lambda_p-\delta_0}{\lambda_1(\lambda_1-\lambda_p+1)}\notag\\
			\ge&\sum\limits_{p>l}\frac{S_k^{pp}u_{1pp}^2}{\lambda_1}\frac{(2-2\delta_0)\lambda_1+\delta_0\lambda_p}{\lambda_1(\lambda_1-\lambda_p+1)}\notag\\
			\ge&\sum\limits_{p>l}\frac{S_k^{pp}u_{1pp}^2}{\lambda_1}\frac{(2-2\delta_0)\lambda_1+\delta_0\left(-\frac{n-k}{k}\lambda_1\right)}{\lambda_1(\lambda_1-\lambda_p+1)}\notag\\
			=&\sum\limits_{p>l}\frac{S_k^{pp}u_{1pp}^2}{\lambda_1}\frac{[2k-(n+k)\delta_0]\frac{\lambda_1}{k}}{\lambda_1(\lambda_1-\lambda_p+1)}\notag\\
			=&0.
		\end{align}
		Plugging \eqref{4.9} \eqref{4.10} \eqref{4.11} into \eqref{4.8}, we have
		\begin{align}\label{4.12}
			&-\frac{1}{\lambda_1}\sum\limits_{p\ne q}S_k^{pp,qq}u_{pp1}u_{qq1}+2\sum\limits_{p>1}\frac{S_k^{pp}u_{1pp}^2}{\lambda_1(\lambda_1-\widetilde{\lambda_p})}\notag\\
			\ge&-(1-\vartheta+\frac{\vartheta}{\delta})\frac{C\lambda_1}{f}+\frac{S_ku_{111}^2}{\lambda_1^3}\notag\\
			\ge&-C\lambda_1+\frac{S_k^{11}u_{111}^2}{\lambda_1^2}+\frac{S_k(\lambda|1)u_{111}^2}{\lambda_1^3}.
		\end{align}
		Combining \eqref{4.12} and \eqref{4.7}, we have
		\begin{align}\label{4.13}
			0\ge&\frac{C\beta}{u}-\frac{\beta S_k^{11}u_1^2}{u^2}+\frac{S_k(\lambda|1)u_{111}^2}{\lambda_1^3}\notag\\
			&+\frac{a}{2}\sum\limits_iS_k^{ii}u_{ii}^2+\frac{A}{2}\sum\limits_iS_k^{ii}-C(A+a+\frac{1}{-u}+\lambda_1).
		\end{align}
		\textbf{Case 1}$\lambda_{k}> 0.$\\
		By Lemma \ref{lem2.2} , we have
		\begin{align}
			\sigma_k(\lambda|1)
			&\ge -C\lambda_2\cdots\lambda_k|\lambda_n|-C\lambda_2\cdots\lambda_{k-1}
			|\lambda_{n-1}||\lambda_n|-\cdots\notag\\
			&\ge -C_3\lambda_2\cdots\lambda_k\cdot K\notag,
		\end{align}
		\begin{align}
			\sigma_{k-1}(\lambda|1)
			&\ge \lambda_{2}\cdots\lambda_{k} -C\lambda_2\cdots\lambda_{k-1}|\lambda_n|-C\lambda_2\cdots\lambda_{k-2}
			|\lambda_{n-1}||\lambda_n|-\cdots\notag\\
			&\ge -C_4\lambda_2\cdots\lambda_{k-1}\cdot K\notag.
		\end{align}
		Using \eqref{4.1}, choosing $\lambda_1$ and $(-u)\lambda_1$ sufficiently large, we have
		\begin{align*}
			\frac{u_{111}^2}{\lambda_1^2}&\le \frac{3\beta^2 u_1^2}{u^2}+3a^2u_1^2u_{11}^2+3A^2x_1^2\\
			&\le \frac{C\beta^2}{u^2}+Ca^2u_{11}^2+CA^2\le 3C_5a^2\lambda_1^2.
		\end{align*}
		Hence
		\begin{align}\label{4.14}
			\frac{S_k(\lambda|1)u_{111}^2}{\lambda_1^3}=&(\sigma_{k-1}(\lambda|1)+\alpha\sigma_{k}(\lambda|1))\frac{1}{\lambda_1}\frac{u_{111}^2}{\lambda_1^2}\notag\\
			\ge&(-\alpha C_3\lambda_2\cdots\lambda_k\cdot K-C_4\lambda_2\cdots\lambda_{k-1}\cdot K)\frac{1}{\lambda_1}\cdot 3C_5a^2\lambda_1^2\notag\\
			\ge&-3C_6Ka^2\lambda_1\cdots\lambda_k-3C_6Ka^2\lambda_1\cdots\lambda_{k-1}.
		\end{align}
		Here $C_6=\max\left\{\alpha C_3C_5, C_4C_5\right\}$.\\
		By direct calculation, we have
		\begin{align}\label{4.15}
			\sum\limits_iS_k^{ii}&=\sum\limits_i[\sigma_{k-2}(\lambda|i)+\alpha\sigma_{k-1}(\lambda|i)]\notag\\
			&=(n-k+2)\sigma_{k-2}(\lambda)+\alpha(n-k+1)\sigma_{k-1}(\lambda)\notag\\
			&\ge(n-k+1)S_{k-1}(\lambda)\notag\\
			&\ge (n-k+1)\frac{\lambda_1\cdots \lambda_{k-2}+\alpha\lambda_1\cdots\lambda_{k-1}}{2}\notag\\
			&\ge \frac{\alpha(n-k+1)}{2}\lambda_1\cdots \lambda_{k-1}.
		\end{align}
		Plugging \eqref{4.14} \eqref{4.15}  into \eqref{4.13}, we obtain
		\begin{align}\label{4.16}
			0\ge&\frac{C\beta}{u}-\frac{\beta S_k^{11}u_1^2}{u^2}-3C_6Ka^2\lambda_1\cdots\lambda_k-3C_6Ka^2\lambda_1\cdots\lambda_{k-1}\notag\\
			&+\frac{a}{2}S_k^{11}u_{11}^2+\frac{A\alpha(n-k+1)}{4}\lambda_1\cdots \lambda_{k-1}-C(A+a+\frac{1}{-u}+\lambda_1)\notag\\
			\ge&\frac{C\beta}{u}-\frac{\beta S_k^{11}u_1^2}{u^2}+\left(\frac{A\alpha(n-k+1)}{4}-3C_6Ka^2-3C_6Ka^2\lambda_k\right)\lambda_1\cdots\lambda_{k-1}\notag\\
			&+\frac{a}{2}S_k^{11}u_{11}^2-C(A+a+2\lambda_1)\notag\\
			\ge&\frac{C\beta}{u}-\frac{\beta S_k^{11}u_1^2}{u^2}+\left(\frac{A\alpha(n-k+1)}{4}-3C_6Ka^2-3C_6Ka^2\lambda_k\right)\lambda_1\cdots\lambda_{k-1}\notag\\
			&+\frac{a}{4}S_k^{11}u_{11}^2+\left(\frac{a\theta f}{4}-2C\right)\lambda_1-C(A+a),
		\end{align}
		we used Lemma \eqref{lem2.3} for the last inequality.\\
		Now we let $A=a^3$ and choose $a$ sufficiently large such that
		$$a>\max\left\{\frac{24C_6C_n^kK^2+12C_6K}{\alpha(n-k+1)},\frac{4(1+2C)}{\theta f}\right\}.$$
		We note $$M=\frac{\frac{a\alpha(n-k+1)}{4}-3C_6K}{3C_6K}.$$
		If $\lambda_{k}\ge M$, then
		\begin{align*}
			S_k(\lambda)&\ge\alpha\sigma_k(\lambda)\notag\\
			&\ge \alpha\lambda_1\cdots\lambda_{k}-C_n^k\alpha\lambda_1\cdots \lambda_{k-1}\cdot K\notag\\
			&=\alpha\lambda_1\cdots \lambda_{k-1}(\lambda_{k}-C_n^kK)\notag\\
			&\ge \alpha\lambda_1\lambda_{k}^{k-2}C_n^kK\notag\\
			&\ge \alpha M^{k-2}C_n^kK\lambda_1.
		\end{align*}
		If $0<\lambda_{k}\le M$, then $\frac{A\alpha(n-k+1)}{4}-3C_6Ka^2-3C_6Ka^2\lambda_k\ge 0$, we have
		\begin{align*}
			0&\ge \frac{C\beta}{u}-\frac{\beta S_k^{11}u_1^2}{u^2}+\frac{a}{4}S_k^{11}u_{11}^2+\left(\frac{a\theta f}{4}-2C\right)\lambda_1-C(A+a)\\
			&\ge \frac{C\beta}{u}-\frac{\beta S_k^{11}u_1^2}{u^2}+\frac{a}{4}S_k^{11}u_{11}^2+\lambda_1-C(A+a),
		\end{align*}
		which implies that desired estimates.\\
		\textbf{Case B}$\lambda_{k}\le 0$.\\
		Imitating the previous discussion, we have
		\begin{align}
			\sigma_k(\lambda|1)
			&\ge -C\lambda_2\cdots\lambda_{k-1}|\lambda_{n-1}||\lambda_n|-C\lambda_2\cdots\lambda_{k-2}
			|\lambda_{n-2}||\lambda_{n-1}||\lambda_n|-\cdots\notag\\
			&\ge -C\lambda_2\cdots\lambda_{k-1}\cdot K^2\notag,
		\end{align}
		\begin{align}
			\sigma_{k-1}(\lambda|1)
			&\ge \lambda_{2}\cdots\lambda_{k-1}|\lambda_{n}| -C\lambda_2\cdots\lambda_{k-2}||\lambda_{n-1}|\lambda_n|-\cdots\notag\\
			&\ge -C\lambda_2\cdots\lambda_{k-1}\cdot K,\notag
		\end{align}
		\begin{align}\label{4.17}
			\frac{S_k(\lambda|1)u_{111}^2}{\lambda_1^3}=&(\sigma_{k-1}(\lambda|1)+\alpha\sigma_{k}(\lambda|1))\frac{1}{\lambda_1}\frac{u_{111}^2}{\lambda_1^2}\notag\\
			\ge&(-C\lambda_2\cdots\lambda_{k-1}\cdot K-C\lambda_2\cdots\lambda_{k-1}\cdot K^2)\frac{1}{\lambda_1}\cdot 3Ca^2\lambda_1^2\notag\\
			\ge&-3CK^2a^2\lambda_1\cdots\lambda_{k-1}.
		\end{align}
		Plugging \eqref{4.15} \eqref{4.17} into \eqref{4.13}, and noticed that $A=a^3$, choose $a$ sufficiently large, we have
		\begin{align*}
			0\ge&\frac{C\beta}{u}-\frac{\beta S_k^{11}u_1^2}{u^2}-3CK^2a^2\lambda_1\cdots\lambda_{k-1}+\frac{a}{2}S_k^{11}u_{11}^2\notag\\
			&+\frac{A\alpha(n-k+1)}{4}\lambda_1\cdots \lambda_{k-1}-C(A+a+\frac{1}{-u}+\lambda_1)\notag\\
			\ge&\frac{C\beta}{u}-\frac{\beta S_k^{11}u_1^2}{u^2}+\left(\frac{A\alpha(n-k+1)}{4}-3CK^2a^2\right)\lambda_1\cdots\lambda_{k-1}\notag\\
			&+\frac{a}{2}S_k^{11}u_{11}^2-C(A+a+2\lambda_1)\\
			\ge& \frac{C\beta}{u}-\frac{\beta S_k^{11}u_1^2}{u^2}+\frac{a}{2}S_k^{11}u_{11}^2-C(A+a+2\lambda_1)\\
			\ge&\frac{C\beta}{u}-\frac{\beta S_k^{11}u_1^2}{u^2}+\frac{a}{4}S_k^{11}u_{11}^2+\left(\frac{a\theta_0}{4}-C\right)\lambda_1.
		\end{align*}
		Hence, we obtain Lemma \ref{lem4.2}.
	\end{proof}
	We now complete the proof of Theorem 1.1 (a). Set $\delta_1=\frac{1}{3}$. By Lemma \ref{lem4.2}, there exists constant $\delta_2$ such that if $\lambda_2\le\delta_2\lambda_1$, then $(-u)^{\beta}\lambda_1\le C$. If $\lambda_2>\delta_2\lambda_1$, using Lemma \ref{lem4.2} again, there exists $\delta_3$ such that if $\lambda_3\le\delta_3\lambda_1$, then $(-u)^{\beta}\lambda_1\le C$. Keep repeating this process and we will obtain
	$(-u)^{\beta}\lambda_1\le C$ or $\lambda_i>\delta_i\lambda_1$, $i=1,2,\cdots,k-1$. For the least case, by $\sigma_{k}(\lambda)\ge -G$ , we have
	\begin{align*}
		S_k&=\sigma_{k-1}+\alpha\sigma_k\\
		&\ge \lambda_1\cdots\lambda_{k-1}-CK\lambda_1\cdots\lambda_{k-2}-G\alpha\\
		&= \lambda_1\cdots\lambda_{k-2}(\lambda_{k-1}-CK)-G\alpha\\
		&\ge \lambda_1\cdots\lambda_{k-2}(\delta_{k-1}\lambda_1-CK)-G\alpha\\
		&\ge C\delta_2\delta_3\cdots\delta_{k-1}\lambda_1^{k-1}-G\alpha,
	\end{align*}
	we complete the proof of Theorem 1.1.
	
	\section{Pogorelov type $C^2$ estimates for $S_n$ equations}
	In this section, we will prove Theorem 1.2. we consider the following test function.
	$$P(x)=\ln\widetilde{\lambda}_1+\beta\ln(-u)+\frac{a}{2}|Du|^2,$$
	where $\beta$ and $a$ are constants which will be determined later. Imitating the previous calculation, we have

	\begin{align}\label{5.1}
		0\ge&\frac{C\beta}{u}-\frac{\beta S_n^{11}u_1^2}{u^2}-\frac{1}{\lambda_1}\sum\limits_{p\ne q}S_n^{pp,qq}u_{pp1}u_{qq1}+\sum\limits_{i>m}\frac{2(S_n^{ii}-S_n^{11})u_{11i}^2}{\lambda_1(\lambda_1-\lambda_i)}\notag\\
		&-C(a+\frac{1}{-u}+\lambda_1)+2\sum\limits_{p>1}\frac{S_n^{11}u_{11p}^2}{\lambda_1(\lambda_1-\widetilde{\lambda_p})}+2\sum\limits_{p>1}\frac{S_n^{pp}u_{1pp}^2}{\lambda_1(\lambda_1-\widetilde{\lambda_p})}-\frac{S_n^{11}u_{111}^2}{\lambda_1^2}\notag\\
		&-\frac{5}{4}\sum\limits_{i>1}\frac{S_n^{ii}u_{11i}^2}{\lambda_1^2}+\frac{a}{2}\sum\limits_iS_n^{ii}u_{ii}^2.
	\end{align}
Since $\lambda=(\lambda_1,\cdots,\lambda_n)\in\widetilde{\Gamma}_n$, we have
$$S_n^{11,22,\cdots,n-1,n-1}(\lambda)=S_1(\lambda|12\cdots n-1)=1+\alpha\lambda_n>0.$$
Thus $\lambda_i>-\frac{1}{\alpha},i=1,2,\cdots,n$. Combining with Lemma \ref{lem4.1}, we still have,
	\begin{align}\label{5.2}
		&\sum\limits_{i>m}\frac{2(S_n^{ii}-S_n^{11})u_{11i}^2}{\lambda_1(\lambda_1-\lambda_i)}+2\sum\limits_{p>1}\frac{S_n^{11}u_{11p}^2}{\lambda_1(\lambda_1-\widetilde{\lambda_p})}-\frac{5}{4}\sum\limits_{i>1}\frac{S_n^{ii}u_{11i}^2}{\lambda_1^2}\notag\\
		&\ge \sum\limits_{p>m}\frac{S_n^{11}u_{11p}^2}{\lambda_1}\frac{\frac{3}{5}\lambda_1+\lambda_p-1}{\lambda_1(\lambda_1-\lambda_p+1)}\ge0.
	\end{align}
	Plugging \eqref{5.2} into \eqref{5.1}, we have
	\begin{align}\label{5.3}
		0\ge&\frac{C\beta}{u}-\frac{\beta S_n^{11}u_1^2}{u^2}-\frac{1}{\lambda_1}\sum\limits_{p\ne q}S_n^{pp,qq}u_{pp1}u_{qq1}+2\sum\limits_{p>1}\frac{S_n^{pp}u_{1pp}^2}{\lambda_1(\lambda_1-\widetilde{\lambda_p})}\notag\\
		&-\frac{S_n^{11}u_{111}^2}{\lambda_1^2}+\frac{a}{2}\sum\limits_iS_n^{ii}u_{ii}^2-C(a+\frac{\beta}{-u}+\lambda_1).
	\end{align}
	Using Lemma \ref{lem3.3}, we have $\exists K>0$, such that
	\begin{equation}\label{5.4}
		-\frac{1}{\lambda_1}\sum\limits_{p\ne q}S_n^{pp,qq}u_{pp1}u_{qq1}+\frac{K}{\lambda_1}\left(\sum\limits_jS_n^{jj}(\lambda)u_{jj1}\right)^2-\frac{S_n^{11}u_{111}^2}{\la_1^2}+\frac{5}{4}\sum\limits_{j> 1}\frac{S_n^{jj}u_{jj1}^2}{\la_1^2}\ge0,
	\end{equation}
we also have
	\begin{equation}\label{5.5}
		\left(\sum\limits_iS_k^{ii}u_{ii1}\right)^2=\left([S_k(D^2u)]_{x_1}\right)^2=(f_{x_1}+f_uu_1+f_{u_1}u_{11})^2\le C\la_1^2.
	\end{equation}
	Plugging \eqref{5.4} \eqref{5.5} into \eqref{5.3}, we have
	\begin{align}\label{6.6}
		0\ge&\frac{C\beta}{u}-\frac{\beta S_n^{11}u_1^2}{u^2}+2\sum\limits_{p>1}\frac{S_n^{pp}u_{1pp}^2}{\lambda_1(\lambda_1-\widetilde{\lambda_p})}-\frac{5}{4}\sum\limits_{j> 1}\frac{S_n^{jj}u_{jj1}^2}{\la_1^2}\notag\\
		&+\frac{a}{2}\sum\limits_iS_n^{ii}u_{ii}^2-C(a+\frac{\beta}{-u}+\lambda_1).
	\end{align}
	Note that
	\begin{align*}
		2\sum\limits_{p>1}\frac{S_n^{pp}u_{1pp}^2}{\lambda_1(\lambda_1-\widetilde{\lambda_p})}-\frac{5}{4}\sum\limits_{j> 1}\frac{S_n^{jj}u_{jj1}^2}{\la_1^2}&=\sum\limits_{p>1}\frac{S_n^{pp}u_{1pp}^2}{\lambda_1}\left(\frac{2}{\la_1-\widetilde{\lambda_p}}-\frac{5}{4\la_1}\right)\\
		&=\sum\limits_{p>1}\frac{S_n^{pp}u_{1pp}^2}{\lambda_1}\frac{3\la_1+5\widetilde{\lambda_p}}{4\la_1(\la_1-\widetilde{\lambda_p})}\ge0.
	\end{align*}
	Hence
	\begin{align*}
		0\ge&\frac{C\beta}{u}-\frac{\beta S_n^{11}u_1^2}{u^2}
		+\frac{a}{2}\sum\limits_iS_n^{ii}u_{ii}^2-C(a+\frac{\beta}{-u}+\lambda_1)\\
		\ge&\frac{C\beta}{u}-\frac{\beta S_n^{11}u_1^2}{u^2}+\frac{a}{2}S_n^{11}u_{11}^2-C(a+\la_1).
	\end{align*}
	Using Lemma \eqref{lem2.3} (ii), we have
	$$S_n^{11}\lambda_1^2\ge \theta S_n\lambda_1\ge\theta_0\lambda_1,$$
	Hence,
	$$-\frac{C\beta}{u}+\frac{\beta S_n^{11}u_1^2}{u^2}\ge\frac{a}{2}S_n^{11}\lambda_1^2-C(a+\lambda_1)\ge\frac{a}{4}S_n^{11}\lambda_1^2+\left(\frac{a\theta_0}{4}-C\right)\lambda_1,$$
	choosing $a$ sufficiently large. We complete our proof.

	\bibliographystyle{plain}
	\bibliography{sumref.bib}

\begin{thebibliography}{10}

\bibitem{Ball}
J.~M. Ball.
\newblock Differentiability properties of symmetric and isotropic functions.
\newblock {\em Duke Math. J.}, 51(3):699--728, 1984.

\bibitem{Bao}
Jiguang Bao, Jiechen Qiang, Zhongwei Tang, and Cong Wang.
\newblock Interior estimates of derivatives and a {L}iouville type theorem for
  parabolic {$k $}-{H}essian equations.
\newblock {\em Commun. Pure Appl. Anal.}, 22(8):2466--2480, 2023.

\bibitem{CNS1}
L.~Caffarelli, L.~Nirenberg, and J.~Spruck.
\newblock The {D}irichlet problem for nonlinear second-order elliptic
  equations. {I}. {M}onge-{A}mp\`ere equation.
\newblock {\em Comm. Pure Appl. Math.}, 37(3):369--402, 1984.

\bibitem{CNS3}
L.~Caffarelli, L.~Nirenberg, and J.~Spruck.
\newblock The {D}irichlet problem for nonlinear second-order elliptic
  equations. {III}. {F}unctions of the eigenvalues of the {H}essian.
\newblock {\em Acta Math.}, 155(3-4):261--301, 1985.

\bibitem{CDH}
Chuanqiang Chen, Weisong Dong, and Fei Han.
\newblock Interior {H}essian estimates for a class of {H}essian type equations.
\newblock {\em Calc. Var. Partial Differential Equations}, 62(2):Paper No. 52,
  15, 2023.

\bibitem{CTX2}
Li~Chen, Qiang Tu, and Ni~Xiang.
\newblock Pogorelov type estimates for a class of {H}essian quotient equations.
\newblock {\em J. Differential Equations}, 282:272--284, 2021.

\bibitem{CX}
Li~Chen and Ni~Xiang.
\newblock Rigidity theorems for the entire solutions of 2-{H}essian equation.
\newblock {\em J. Differential Equations}, 267(9):5202--5219, 2019.

\bibitem{CLTX}
Xiaojuan Chen, Wei Lu, Qiang Tu, and Ni~Xiang.
\newblock Pogorelov estimates for a class of fully nonlinear equations.
\newblock {\em Nonlinear Anal.}, 212:Paper No. 112482, 12, 2021.

\bibitem{CTX}
Xiaojuan Chen, Qiang Tu, and Ni~Xiang.
\newblock Pogorelov estimates for semi-convex solutions of {$k$}-curvature
  equations.
\newblock {\em Proc. Amer. Math. Soc.}, 152(7):2923--2936, 2024.

\bibitem{CW}
Kai-Seng Chou and Xu-Jia Wang.
\newblock A variational theory of the {H}essian equation.
\newblock {\em Comm. Pure Appl. Math.}, 54(9):1029--1064, 2001.

\bibitem{Chu}
Jianchun Chu.
\newblock A simple proof of curvature estimate for convex solution of
  {$k$}-{H}essian equation.
\newblock {\em Proc. Amer. Math. Soc.}, 149(8):3541--3552, 2021.

\bibitem{JC}
Jianchun Chu and Heming Jiao.
\newblock Curvature estimates for a class of {H}essian type equations.
\newblock {\em Calc. Var. Partial Differential Equations}, 60(3):Paper No. 90,
  18, 2021.

\bibitem{D1}
Hongjie Dong.
\newblock Hessian equations with elementary symmetric functions.
\newblock {\em Comm. Partial Differential Equations}, 31(7-9):1005--1025, 2006.

\bibitem{DWS2}
Weisong Dong.
\newblock Curvature estimates for {$p$}-convex hypersurfaces of prescribed
  curvature.
\newblock {\em Rev. Mat. Iberoam.}, 39(3):1039--1058, 2023.

\bibitem{DWS1}
Weisong Dong.
\newblock The {D}irichlet problem for prescribed curvature equations of
  {$p$}-convex hypersurfaces.
\newblock {\em Manuscripta Math.}, 174(3-4):785--806, 2024.

\bibitem{GB1}
Bo~Guan.
\newblock The {D}irichlet problem for {H}essian equations on {R}iemannian
  manifolds.
\newblock {\em Calc. Var. Partial Differential Equations}, 8(1):45--69, 1999.

\bibitem{GG}
Bo~Guan and Pengfei Guan.
\newblock Convex hypersurfaces of prescribed curvatures.
\newblock {\em Ann. of Math. (2)}, 156(2):655--673, 2002.

\bibitem{GLL}
Pengfei Guan, Junfang Li, and Yanyan Li.
\newblock Hypersurfaces of prescribed curvature measure.
\newblock {\em Duke Math. J.}, 161(10):1927--1942, 2012.

\bibitem{GLM}
Pengfei Guan, Changshou Lin, and Xi-Nan Ma.
\newblock The existence of convex body with prescribed curvature measures.
\newblock {\em Int. Math. Res. Not. IMRN}, (11):1947--1975, 2009.

\bibitem{GRW}
Pengfei Guan, Changyu Ren, and Zhizhang Wang.
\newblock Global {$C^2$}-estimates for convex solutions of curvature equations.
\newblock {\em Comm. Pure Appl. Math.}, 68(8):1287--1325, 2015.

\bibitem{GZ}
Pengfei Guan and Xiangwen Zhang.
\newblock A class of curvature type equations.
\newblock {\em Pure Appl. Math. Q.}, 17(3):865--907, 2021.

\bibitem{HLB}
Reese Harvey and H.~Blaine Lawson, Jr.
\newblock Calibrated geometries.
\newblock {\em Acta Math.}, 148:47--157, 1982.

\bibitem{HPX}
Yan He, Cen Pan, and Ni~Xiang.
\newblock A rigidity theorem for parabolic 2-{H}essian equations.
\newblock {\em Proc. Amer. Math. Soc.}, 150(9):3821--3830, 2022.

\bibitem{HSX}
Yan He, Haoyang Sheng, Ni~Xiang, and Jiannan Zhang.
\newblock A {P}ogorelov estimate and a {L}iouville-type theorem to parabolic
  {$k$}-{H}essian equations.
\newblock {\em Commun. Contemp. Math.}, 24(4):Paper No. 2150001, 21, 2022.

\bibitem{Ivo2}
N.~M. Ivochkina.
\newblock Solution of the {D}irichlet problem for equations of {$m$}th order
  curvature.
\newblock {\em Mat. Sb.}, 180(7):867--887, 991, 1989.

\bibitem{Ivo1}
N.~M. Ivochkina.
\newblock The {D}irichlet problem for the curvature equation of order {$m$}.
\newblock {\em Algebra i Analiz}, 2(3):192--217, 1990.

\bibitem{JT}
Feida Jiang and Neil~S. Trudinger.
\newblock On {P}ogorelov estimates in optimal transportation and geometric
  optics.
\newblock {\em Bull. Math. Sci.}, 4(3):407--431, 2014.

\bibitem{JJ}
Heming Jiao and Yang Jiao.
\newblock The {P}ogorelov {E}stimates for {D}egenerate {C}urvature {E}quations.
\newblock {\em Int. Math. Res. Not. IMRN}, (18):12504--12529, 2024.

\bibitem{JL}
Qinian Jin and Yanyan Li.
\newblock Starshaped compact hypersurfaces with prescribed {$k$}-th mean
  curvature in hyperbolic space.
\newblock {\em Discrete Contin. Dyn. Syst.}, 15(2):367--377, 2006.

\bibitem{K1}
N.~V. Krylov.
\newblock On the general notion of fully nonlinear second-order elliptic
  equations.
\newblock {\em Trans. Amer. Math. Soc.}, 347(3):857--895, 1995.

\bibitem{LRW}
Chunhe Li, Changyu Ren, and Zhizhang Wang.
\newblock Curvature estimates for convex solutions of some fully nonlinear
  {H}essian-type equations.
\newblock {\em Calc. Var. Partial Differential Equations}, 58(6):Paper No. 188,
  32, 2019.

\bibitem{LiRW}
Ming Li, Changyu Ren, and Zhizhang Wang.
\newblock An interior estimate for convex solutions and a rigidity theorem.
\newblock {\em J. Funct. Anal.}, 270(7):2691--2714, 2016.

\bibitem{LMZ}
Chenyang Liu, Jing Mao, and Yating Zhao.
\newblock Pogorelov type estimates for a class of {H}essian quotient equations
  in {L}orentz-{M}inkowski space {$\Bbb R_1^{n+1}$}.
\newblock {\em J. Differential Equations}, 327:212--225, 2022.

\bibitem{LT}
Jiakun Liu and Neil~S. Trudinger.
\newblock On {P}ogorelov estimates for {M}onge-{A}mp\`ere type equations.
\newblock {\em Discrete Contin. Dyn. Syst.}, 28(3):1121--1135, 2010.

\bibitem{LR}
Yue Liu and Changyu Ren.
\newblock Pogorelov type {$C^2$} estimates for sum {H}essian equations and a
  rigidity theorem.
\newblock {\em J. Funct. Anal.}, 284(1):Paper No. 109726, 32, 2023.

\bibitem{Lu}
Siyuan Lu.
\newblock Curvature estimates for semi-convex solutions of {H}essian equations
  in hyperbolic space.
\newblock {\em Calc. Var. Partial Differential Equations}, 62(9):Paper No. 257,
  23, 2023.

\bibitem{Mei}
Xinqun Mei.
\newblock Interior {$C^2$} estimates for the {H}essian quotient type equation.
\newblock {\em Proc. Amer. Math. Soc.}, 151(9):3913--3924, 2023.

\bibitem{P1}
A.~V. Pogorelov.
\newblock {\em The {M}inkowski multidimensional problem}.
\newblock Scripta Series in Mathematics. V. H. Winston \& Sons, Washington, DC;
  Halsted Press [John Wiley \& Sons], New York-Toronto-London, 1978.
\newblock Translated from the Russian by Vladimir Oliker, Introduction by Louis
  Nirenberg.

\bibitem{RW1}
Changyu Ren and Zhizhang Wang.
\newblock On the curvature estimates for {H}essian equations.
\newblock {\em Amer. J. Math.}, 141(5):1281--1315, 2019.

\bibitem{RW2}
Changyu Ren and Zhizhang Wang.
\newblock The global curvature estimate for the {$n-2$} {H}essian equation.
\newblock {\em Calc. Var. Partial Differential Equations}, 62(9):Paper No. 239,
  50, 2023.

\bibitem{SUW}
Weimin Sheng, John Urbas, and Xu-Jia Wang.
\newblock Interior curvature bounds for a class of curvature equations.
\newblock {\em Duke Math. J.}, 123(2):235--264, 2004.

\bibitem{SXia}
Weimin Sheng and Shucan Xia.
\newblock Interior curvature bounds for a type of mixed {H}essian quotient
  equations.
\newblock {\em Math. Eng.}, 5(2):Paper No. 040, 27, 2023.

\bibitem{SX}
Joel Spruck and Ling Xiao.
\newblock A note on star-shaped compact hypersurfaces with prescribed scalar
  curvature in space forms.
\newblock {\em Rev. Mat. Iberoam.}, 33(2):547--554, 2017.

\bibitem{Wang}
Xu-Jia Wang.
\newblock The {$k$}-{H}essian equation.
\newblock In {\em Geometric analysis and {PDE}s}, volume 1977 of {\em Lecture
  Notes in Math.}, pages 177--252. Springer, Dordrecht, 2009.

\bibitem{Y1}
Fengrui Yang.
\newblock Prescribed curvature measure problem in hyperbolic space.
\newblock {\em Comm. Pure Appl. Math.}, 77(1):863--898, 2024.

\bibitem{YY}
Yu~Yuan.
\newblock A monotonicity approach to {P}ogorelov's {H}essian estimates for
  {M}onge-{A}mp\`ere equation.
\newblock {\em Math. Eng.}, 5(2):Paper No. 037, 6, 2023.

\bibitem{Zhou2}
Jundong Zhou.
\newblock The interior gradient estimate for a class of mixed {H}essian
  curvature equations.
\newblock {\em J. Korean Math. Soc.}, 59(1):53--69, 2022.

\bibitem{Zhou1}
Jundong Zhou.
\newblock Curvature estimates for a class of {H}essian quotient type curvature
  equations.
\newblock {\em Calc. Var. Partial Differential Equations}, 63(4):Paper No. 88,
  21, 2024.

\end{thebibliography}

\end{document}